\documentclass{amsart}

\usepackage{graphicx}
\usepackage[letterpaper, margin=1in]{geometry}
\usepackage{appendix}
\usepackage{amsfonts}
\usepackage{amsmath}
\usepackage{amssymb}
\usepackage{bbm}
\usepackage{bm}
\usepackage{mathrsfs}
\usepackage[inline, shortlabels]{enumitem}
\usepackage{verbatim}
\usepackage{setspace}
\usepackage{color}
\usepackage{pdfsync}
\usepackage{enumitem}
\usepackage{bbm}

\newcommand{\eps}{\varepsilon}
\newcommand{\vp}{\varphi}

\newcommand{\N}{\mathbb{N}}

\newcommand{\R}{\mathbb{R}}

\renewcommand{\P}{\mathbb{P}}
\renewcommand{\L}{\mathbb{L}}

\newcommand{\bU}{\mathbb{U}}

\newcommand{\Uc}{\mathcal{U}}
\newcommand{\Di}{\mathbb{D}_{i}}
\newcommand{\D}{\mathbb{D}}
\newcommand{\DT}{\mathbb{D}_{T}}
\newcommand{\unco}{\text{unco}}
\newcommand{\co}{\text{co}}

\newcommand{\as}{\mbox{-a.s.}}

\def\essinf{{\rm ess}\!\inf\limits}
\def\esssup{{\rm ess}\!\sup\limits}

\def\be{\begin{eqnarray}}
\def\ee{\end{eqnarray}}
\def\beq{\begin{equation}}
\def\eeq{\end{equation}}

\newtheorem{assumption}{Assumption}[section]
\newtheorem{theorem}{Theorem}[section]
\newtheorem{corollary}{Corollary}[section]
\newtheorem{proposition}{Proposition}[section]
\newtheorem{lemma}{Lemma}[section]

\theoremstyle{definition}
\newtheorem{definition}{Definition}[section]

\theoremstyle{remark}
\newtheorem{remark}{Remark}[section]

\numberwithin{equation}{section}

\begin{document}

\title{On the controller-stopper problems with controlled jumps}


\author{Erhan Bayraktar} \thanks{E. Bayraktar is supported in part by the National Science Foundation under grant DMS-1613170 and the Susan M. Smith Professorship.}
\address{Erhan Bayraktar \\ University of Michigan, Ann Arbor }
\curraddr{}
\email{erhan@umich.edu}
\thanks{}

\author{Jiaqi Li}
\address{Jiaqi Li \\ Goldman Sachs}
\curraddr{}
\email{lijiaqi@umich.edu}
\thanks{}


\date{}

\begin{abstract}
 We analyze the continuous time zero-sum and cooperative controller-stopper games of Karatzas and Sudderth [Annals of Probability, 2001], Karatzas and Zamfirescu [Annals of Probability, 2008] and Karatzas and Zamfirescu [Applied Mathematics and Optimization, 2005] when the volatility of the state process is controlled as in Bayraktar and Huang [SIAM Journal on Control and Optimization, 2013] but additionally when the state process has controlled jumps.
 We perform this analysis by first resolving the stochastic target problems (of Soner and Touzi [SIAM Journal on Control and Optimization, 2002; Journal of European Mathematical Society, 2002]) with a cooperative or a non-cooperative stopper and then embedding the original problem into the latter set-up. Unlike in Bayraktar and Huang [SIAM Journal on Control and Optimization, 2013] our analysis relies crucially on the Stochastic Perron method of Bayraktar and S\^{i}rbu [SIAM Journal on Control and Optimization, 2013] but not the dynamic programming principle, which is difficult to prove directly for games. 
\end{abstract}

\maketitle




\section{Introduction}

The zero-sum stochastic games between a controller (who controls the state dynamics) and a stopper (who chooses the termination time of the game) were introduced by \cite{MR1481781} in discrete time and were then resolved by \cite{KS2001} for one-dimensional diffusions. Later, \cite{KZ2008} and \cite{BKY2010} considered this problem when the underlying diffusion is multi-dimensional but when only the drift is controlled.  See also \cite{MR2746174, MR2732930, ECTA:ECTA939, MR3023890, MR3476637}. This problem was later on solved for the case when the volatility is controlled and can be degenerate in \cite{Bayraktar_Huang} and was further generalized in \cite{MR3267150, MR3375882, BS-DG-2016}. 
The cooperative version of the game has received much attention as well. The general theoretical results on cooperative controller-stopper problems (also called problems of stochastic control with discretionary stopping) were obtained in \cite{Krylov1980,K1981,BL1982,MS1996,BC2004}. Later, a martingale treatment of the controller-stopper problems was developed in \cite{KZ2006}. Later this was analyzed in a more general case when the volatility is also controlled in \cite{MR3231620}. 

In this paper, we will generalize these two types of stopping games to the case in which the controller can control the jumps. We will prove our results by embedding these two problems into what is called ``stochastic target problems" \emph{with a stopper} and solving a more general problem. These problems are more difficult because the goal is to derive the state process to a target almost surely. The original stochastic target problems were introduced by \cite{DP_FOR_STP_AND_G_FLOW, SONER_TOUZI_STG} as a generalization of the super-replication problem in Mathematical Finance, in which the goal is to drive the controlled process to a given target at a given terminal time. There is an extensive literature on this subject which considered these problems with an increasing level of generality, see e.g. \cite{Bruno_jump_diffusion, Bouchard_Elie_Touzi_ControlledLoss, MR2585143, STG_CONTROLLED_LOSS, Moreau}. A survey of these results are given in Touzi's book \cite{Touzi_book}.
 In the two versions analyzed in this paper, the terminal time is a stopping time, which is either chosen by the controller (cooperative version), or the controller has to be robust against the choice of the stopping time (uncooperative version).
We use the jump diffusion model presented in \cite{BayraktarLi-Jump} (also see \cite{Moreau}) for the evolution of the state process and first analyze the target problems, one of which involves a cooperative stopper (Section \ref{sec:subhedging}), and the other a non-cooperative stopper who might play against the controller in a non-anticipative way (Section \ref{sec:superhedging}).  In each of these target problems,  we use stochastic Perron's method (of \cite{Bayraktar_and_Sirbu_SP_HJBEqn}), instead of relying on the geometric dynamic programming principle (see \cite{DP_FOR_STP_AND_G_FLOW}) to create a viscosity sub-solution and super-solution to its associated Hamilton-Jacobi-Bellman (HJB) equation. Then by establishing an embedding result similar in spirit to \cite{Bouchard_Equivalence} between the controller-stopper problems and the stochastic target problems and assuming that a comparison principle holds, we show that the value functions of the original controller-stopper problems are unique viscosity solutions of the corresponding HJB equations. It is interesting to note that in the cooperative controller-stopper problem we observe the \emph{face-lifting phenomenon}, i.e., there is a possible discontinuity at the terminal time, whereas there is no such occurrence in the non-cooperative version. In fact, this discontinuity is exactly characterized. The observation that there is discontinuity at the terminal time in the cooperative controller-stopper game goes back to \cite{Krylov1980}, but there the magnitude was not identified. It is also worth recording that the face-lifting occurs in both of the corresponding stochastic target problems, but the reasons for the discontinuity are different. 

Using the geometric dynamic programming principle, \cite{MR2585143} also considered the non-cooperative version of the stochastic target problem in the context of pricing American options with investment constraints in a Brownian diffusion type financial market. Our focus on the other hand is the embedding result in the spirit of \cite{Bouchard_Equivalence} and resolving both the cooperative and zero-sum controller-stopper games of Karatzas, Sudderth and Zamfirescu. Moreover, our results rely on the stochastic Perron's method of \cite{Bayraktar_and_Sirbu_SP_HJBEqn} in generating the sub- and super-solutions of the corresponding HJB equations without relying on the dynamic programming principle and skipping the technical difficulties due to measurability issues. In general, dynamic programming principle for stochastic differential games is quite complicated, see e.g. \cite{Bayraktar_Huang, Bouchard_Nutz_TargetGames, MR997385}. Stochastic Perron's method (a verification type result without smoothness), by working with appropriate envelopes instead of the value function itself, avoids having to prove a dynamic programming principle altogether.
This method is similar in spirit to the Perron's construction of viscosity solutions presented in \cite{UserGuide}. The crucial difference is that stochastic Perron's method constructs the viscosity sub- and super solution to envelope the value function of the control problem. See \cite{Bayraktar_and_Sirbu_SP_LinearCase, Bayraktar_and_Sirbu_SP_DynkinGames, MR3274519, MR3217159, MR3295681, Sirbu_SP_elementary_strategy, BayraktarLi, MR3535885, BCP14, BCS2016, 2016arXiv160807498B} for some recent results on the applications of this method.

The rest of the paper is organized as follows: In Section \ref{sec:prob},  the two stochastic target problems and their associated HJB equations are introduced. In Sections \ref{sec:superhedging}, using the stochastic Perron's method we will analyze the stochastic target problem in which the controller needs to be robust with respect to the choice of the stopping time by which the target needs to be reached. In Section~\ref{subsec:equivalence}, we establish the relationship between this problem and the zero-sum controller-stopper game. Using the results of the previous section and using a comparison principle, we demonstrate that the value function of the zero-sum controller-stopper game is the unique viscosity solution of the corresponding HJB equation. Sections~\ref{sec:subhedging} and \ref{sec:equivalence-subhedging} do the same for the cooperative controller-stopper problem. The main results of the paper are Theorems~\ref{thm: optimal control-superhedging}, \ref{thm: optimal control-subhedging}, and their corollaries. However, the results on the stochastic target problems in the auxiliary sections, Sections~\ref{sec:superhedging} and ~\ref{sec:subhedging}, where the bulk of the technical work is done, contain some new results which we also designated as theorems.
The Appendix contains technical results that are crucial in embedding the controller-stopper problems into stochastic target problems. 

\hfil \newline

\noindent \textbf{Notation.} Throughout this paper, the superscript $^{\top}$ stands for transposition, $|\cdot|$ for the Euclidean norm of a vector in $\R^n$, $] a,b [$ for the open interval in $\R$ from $a$ to $b$.  For a subset of $\mathcal{O}$ of $\R^n$, we denote by Int$(\mathcal{O})$ its interior and by $\text{cl}(\mathcal{O})$ it closure. We also denote the open ball of radius $r>0$ centered at $x\in \R^n$ by $B_r(x)$.  Inequalities and inclusion between random variables and random sets, respectively, are in the almost sure sense unless otherwise stated.

\section{Setting up the Stochastic Target Problems}\label{sec:prob}

We will define the stochastic target problems that we will use as auxiliary tools in establishing the characterization of the controller-stopper problems as the unique viscosity solutions of HJB equations. We will first have to introduce some relevant concepts and notation. These functions and the HJB equations for the target problems will appear at the end of this section. Before we introduce the set-up, we emphasize that the set-up has been used in \cite{Moreau} and \cite{BayraktarLi-Jump}. 

Given a complete probability space $(\Omega, \mathcal{F}, \mathbb{P})$, let $\{\lambda_i(\cdot,de)\}_{i=1}^I$ be a collection of independent integer-valued $E$-marked right-continuous point processes defined on this space. Here, $E$ is a Borel subset of $\mathbb{R}$ equipped with the Borel sigma field $\mathcal{E}$. Let $\lambda=(\lambda_1,\lambda_2,\cdots,\lambda_I)^{\top}$and $W = \{W_s\}_{0\leq s\leq T}$ be a $d$-dimensional Brownian motion defined on the same probability space such that $W$ and $\lambda$ are independent. Given $t\in[0,T]$, let $\mathbb{F}^t=\{\mathcal{F}^t_s, t\leq s\leq T\}$ be $\mathbb{P}$-augmented filtration generated by $W_{\cdot}-W_t$ and $\lambda([0,\cdot],de)-\lambda([0,t],de)$.  By convention, set $\mathcal{F}^t_s = \mathcal{F}^t_t $ for $0\leq s < t$. We will use $\mathcal{T}_t$ to denote the set of $\mathbb{F}^t$-stopping times valued in $[t,T]$. Given $\tau\in\mathcal{T}_t$, the set of $\mathbb{F}^t$-stopping times valued in $[\tau, T]$ will be denoted by $\mathcal{T}_{\tau}$.  
\begin{assumption}\label{assump: lambda intensity kernel}
$\lambda$ satisfies the following:
\begin{enumerate}
\item $\lambda(ds,de)$ has intensity kernel $m(de)ds$ such that $m_{i}$ is a Borel measure on $(E,\mathcal{E})$ for any $i=1,\cdots, I$ and $\hat{m}(E)<\infty$, where $m=(m_1,\cdots,m_I)^{\top}$ and $\hat{m}=\sum_{i=1}^{I} m_{i}$.
 \item $E=\text{supp}(m_i)$ for all $i=1,2,\cdots,I$. Here, $\text{supp}(m_{i}):=\{e\in E:  e\in N_{e}\in T_{E} \implies m_{i}(N_{e})>0\},$
 where $T_{E}$ is the topology on $E$ induced by the Euclidean topology. 
 \item There exists a constant $C>0$ such that 
 $$\mathbb{P}\left(\left\{\hat{\lambda}(\{s\}, E)\leq C \;\;\text{for all}\;\; s\in[0,T]\right\}\right)=1,\;\;\text{where}\;\;\hat{\lambda}=\sum_{i=1}^{I}\lambda_{i}.$$ 
\end{enumerate}
\end{assumption}
The above assumption implies that there are a finite number of jumps during any finite time interval. Let $\tilde{\lambda}(ds,de):=\lambda(ds,de)-m(de)ds$ be the associated compensated random measure.  

Let $\mathcal{U}^t_1$ be the collection of all the $\mathbb{F}^{t}$-predictable processes in $\mathbb{L}^2(\Omega\times[0,T], \mathcal{F}\otimes\mathcal{B}[0,T], \mathbb{P}\otimes \lambda_{L}; U_{1})$, where $\lambda_{L}$ is the Lebesgue measure on $\R$ and $U_{1}\subset \R^{q}$ for some $q\in\N$. Define $\mathcal{U}_2^{t}$ to be the collection of all the maps $\nu_{2}:\Omega\times[0,T]\times E\rightarrow \R^{n}$ which are $\mathcal{P}^{t}\otimes\mathcal{E}$ measurable such that 
\begin{equation}\label{eq:integrablility}
\|\nu_{2}\|_{\mathcal{U}_2^{t}}:=\left(\mathbb{E}\left[\int_{t}^{T}\int_E|\nu_{2}(s,e)|^2 \hat{m}(de)ds\right]\right)^{\frac{1}{2}}<\infty,
\end{equation}
 where $\mathcal{P}^{t}$ is the $\mathbb{F}^t$-predictable sigma-algebra on $\Omega\times[0,T]$. $\nu=(\nu_{1},\nu_{2})\in\mathcal{U}_{0}^{t} := \mathcal{U}_{1}^{t}\times\mathcal{U}^{t}_{2}$ takes value in the set $U:=U_1\times \mathbb{L}^2(E,\mathcal{E}, \hat{m};\R^{n})$.  
Let 
\begin{equation}
\D=[0,T]\times \R^d,\quad \Di=[0,T[\;\times\;\R^d \quad \text{ and } \DT=\{T\}\times \R^d.
\end{equation}
Given $z = (x, y)\in\R^d\times\R $, $t \in [0, T]$ and $\nu\in\mathcal{U}_{0}^{t}$, we consider the stochastic differential equations (SDEs)
\begin{equation}\label{eq: SDEs}
\begin{array}{l}
dX(s)=\mu_{X}(s,X(s),\nu(s))ds+\sigma_{X}(s,X(s),\nu(s))dW_s+\int_{E} \beta(s,X(s-),\nu_1(s),\nu_2(s,e), e)\lambda(ds,de), \vspace{0.07in}\\
dY(s)=\mu_{Y}(s,Z(s),\nu(s))ds+\sigma_{Y}^{\top}(s,Z(s),\nu(s)) dW_s+ \int_{E} b^{\top}(s,Z(s-),\nu_1(s),\nu_2(s,e), e)\lambda(ds,de), 
\end{array}
\end{equation}
with $(X(t), Y(t))=(x,y)$. Here, $Z=(X,Y)$. 
In \eqref{eq: SDEs}, 
 \begin{equation*}
\begin{array}{c}
\mu_X: \D\times U \rightarrow \R^d,\;\;\sigma_X: \D\times U\rightarrow \R^{d\times d},\;\; \beta: \D\times U_1\times\R^n\times E \rightarrow \R^{d\times I}, \\
\mu_Y: \D\times\R\times U \rightarrow \mathbb{R},\;\;\sigma_Y: \D\times\R\times U\rightarrow \R^{d},\;\; b: \D\times\R\times U_1\times\R^n\times E \rightarrow \R^{I}. 
\end{array}
\end{equation*}

Let $\mathcal{U}_{\unco}^{t}$ be the admissible control set for the stochastic target problem with a non-cooperative stopper, which consists of all $\nu \in \mathcal{U}^{t}_{0}$ such that for any compact set $C\subset \R^{d}\times \R$ and $\tau\in\mathcal{T}_{t}$, there exists a constant $K_{\unco}^{C, \nu,\tau}>0$ such that 
\begin{equation}\label{eq:admissibility_super}
\int_{E} b^{\top}(\tau,x, y, \nu_{1}(\tau), \nu_{2}(\tau, e), e ) \lambda(\{\tau\}, de) \geq - K_{\unco}^{C, \nu, \tau}\;\;\text{for all}\;\; (x,y)\in C.
\end{equation}
Let $\mathcal{U}_{\co}^{t}$ be the admissible control set for the stochastic target problem with a cooperative stopper, which consists of all $\nu \in \mathcal{U}^{t}_{0}$ such that for any compact set $C\subset \R^{d}\times \R$ and $\tau\in\mathcal{T}_{t}$, there exists a constant $K_{\co}^{C, \nu, \tau}>0$ such that 
\begin{equation}\label{eq:admissibility_sub}
\int_{E} b^{\top}(\tau,x, y, \nu_{1}(\tau), \nu_{2}(\tau, e), e ) \lambda(\{\tau\}, de) \leq  K_{\co}^{C, \nu,\tau}\;\;\text{for all}\;\; (x,y)\in C.
\end{equation}

\begin{assumption}\label{assump: regu_on_coeff}
Let $z=(x,y)$ and $u=(u_1,u_2)\in U = U_1\times\mathbb{L}^2(E,\mathcal{E},\hat{m};\R^{n})$. We use the notation $\|u\|_{U}:=|u_1|+\|u_2\|_{\hat{m}}$ and  $u(e):=(u_1,u_2(e))$ for the rest of the paper. 
\begin{enumerate}
\item  $\mu_X, \sigma_X$, $\mu_Y$ and $\sigma_Y$ are all continuous; 
\item  $\mu_X, \sigma_X$, $\mu_Y$, $\sigma_Y$ are Lipschitz in $z$ and locally Lipschitz in other variables. In addition,
\begin{equation*}
|\mu_X(t,x,u)|+|\sigma_X(t,x,u)|\leq L(1+|x|+\|u\|_{U}), \;\;|\mu_Y(t,x,y,u)|+|\sigma_Y(t,x,y,u)|\leq L(1+|y|+\|u\|_{U}).
\end{equation*} 
\item $b$ and $\beta$ are Lipschitz and grow linearly in all variables except $e$, but uniformly in $e$.
\end{enumerate}
\end{assumption}
\begin{remark} \label{remark: facts from assumption 1}
Assumptions \ref{assump: lambda intensity kernel} and \ref{assump: regu_on_coeff} guarantee that there exists a unique strong solution $(X_{t,x}^{\nu}, Y_{t,x,y}^{\nu})$ to \eqref{eq: SDEs} for any $\nu\in\Uc_{0}^{t}$. This follows from a simple arguments using Gronwall's Lemma; see e.g. for this result in a similar set-up in \cite{Pham1998}. Also see Lemma 17.1.1 in  \cite{MR3443368}.
\end{remark}

Now, we are ready to introduce the auxiliary stochastic target problems (with a stopper) that we will analyze in this paper. The main problems of zero-sum or cooperative controlled and stopper games will be introduced in Sections 4 and 6.

\subsection{\bf The controller and stopper problems and the HJB operators}\label{subsection: HJB}
Let $g: \R^d\rightarrow \R$ be a continuous function with polynomial growth. The value functions of the two target problems are defined respectively by
\begin{equation}\label{eq: value_function_super}
u_{\unco}(t,x):=\inf\left\{y:\; \exists \nu \in \Uc^t_{\unco} \text{  such that  } Y_{t,x,y}^{\nu}(\rho)\geq g(X_{t,x}^{\nu}(\rho)) \; \mathbb{P}-\text{a.s. for all }\rho\in\mathcal{T}_{t}\right\},
\end{equation}
\begin{equation}\label{eq: value_function_sub}
u_{\co}(t,x):=\sup\left\{y:\; \exists \nu \in \Uc^t_{\co}\text{ and }\rho\in\mathcal{T}_{t} \text{  such that  } Y_{t,x,y}^{\nu}(\rho)\leq g(X_{t,x}^{\nu}(\rho)) \; \mathbb{P}-\text{a.s.}\right\}.
\end{equation}
Denote $b=(b_1,b_2,\cdots,b_I)^{\top}$ and $\beta=(\beta_1,\beta_2,\cdots,\beta_I)$. For a given $\vp \in C(\D)$, we define the relaxed semi-limits 
\begin{equation}\label{eq: HJB operators-super}
H^{*}(\Theta, \vp):=\limsup_{\begin{subarray}{c}\eps\searrow 0,\; \Theta^{'}\rightarrow\Theta \\ \eta\searrow 0,\;  \psi\overset{\text{u.c.}}{\longrightarrow} \vp\end{subarray}} H_{\eps,\eta} (\Theta^{'}, \psi) \;\; \text{and} \;\; H_{*}(\Theta, \vp):=\liminf_{\begin{subarray}{c}\eps\searrow 0,\; \Theta^{'}\rightarrow\Theta \\  \eta\searrow 0,\; \psi \overset{\text{u.c.}}{\longrightarrow} \vp\end{subarray}} H_{\eps, \eta} (\Theta^{'}, \psi), \footnote{The convergence $\psi \overset{\text{u.c.}}{\longrightarrow} \vp$ is understood in the sense that $\psi$ converges uniformly on compact subsets to $\vp$.}
\end{equation}
\begin{equation}\label{eq: HJB operators-sub}
F^{*}(\Theta, \vp):=\limsup_{\begin{subarray}{c}\eps\searrow 0,\; \Theta^{'}\rightarrow\Theta \\ \eta\searrow 0,\;  \psi\overset{\text{u.c.}}{\longrightarrow} \vp\end{subarray}} F_{\eps,\eta} (\Theta^{'}, \psi) \;\; \text{and} \;\; F_{*}(\Theta, \vp):=\liminf_{\begin{subarray}{c}\eps\searrow 0,\; \Theta^{'}\rightarrow\Theta \\  \eta\searrow 0,\; \psi \overset{\text{u.c.}}{\longrightarrow} \vp\end{subarray}} F_{\eps, \eta} (\Theta^{'}, \psi).
\end{equation}
Here, for $\Theta = (t,x,y,p,A)\in\D\times\R\times\R^d\times\mathbb{M}^d \;( \mathbb{M}^d:=\R^{d\times d}$), $\vp\in C(\D)$, $\eps\geq 0$ and $\eta\in[-1,1]$, 
\begin{equation*}
\begin{gathered}
H_{\eps, \eta}(\Theta, \vp):=\sup_{u\in\mathcal{N}_{\eps, \eta}(t,x,y,p,\vp)} \mathbf{L}^u(\Theta),\; F_{\eps, \eta}(\Theta, \vp):=\inf_{u\in\mathcal{M}_{\eps, \eta}(t,x,y,p,\vp)} \mathbf{L}^u(\Theta), \text{where},\\
\begin{array}{c}
\mathbf{L}^u(\Theta):=\mu_{Y}(t,x,y,u) -\mu_{X}^{\top}(t,x,u) p -\frac{1}{2}\text{Tr}[\sigma_X\sigma_X^{\top}(t,x,u)A],\;\;N^{u}(t,x,y,p):=\sigma_Y(t,x,y,u)-\sigma_X^{\top}(t,x,u)p, \\
\Delta^{u,e}(t,x,y,\vp):=\min_{1\leq i \leq I} \{b_i(t,x,y,u(e),e)-\vp(t,x+\beta_i(t,x,u(e),e))+\vp(t,x) \}, \\
\Pi^{u,e}(t,x,y,\vp):=\max_{1\leq i \leq I} \{b_i(t,x,y,u(e),e)-\vp(t,x+\beta_i(t,x,u(e),e))+\vp(t,x) \}, \\
\mathcal{N}_{\eps,\eta}(t,x,y,p, \vp):=\{u\in U: |N^{u}(t,x,y,p)|\leq\eps\text{ and }
\Delta^{u,e}(t,x,y,\vp)\geq\eta\;\text{for}\; \hat{m}-\text{a.s.}\; e\in E \}, \\
\mathcal{M}_{\eps,\eta}(t,x,y,p, \vp):=\{u\in U: |N^{u}(t,x,y,p)|\leq\eps\text{ and }
\Pi^{u,e}(t,x,y,\vp)\leq\eta\;\text{for}\; \hat{m}-\text{a.s.}\; e\in E \},
\end{array}
\end{gathered}
\end{equation*}
where $\hat{m}$ is as in Asumption \ref{assump: lambda intensity kernel}. For our later use, we also define the following:
\begin{equation*}
\begin{array}{c}
J^{u,e}_i(t,x,y,\vp):= b_i(t,x,y,u(e),e)-\vp(t,x+\beta_i(t,x,u(e),e))+\vp(t,x), \nonumber \\
\overline{J}^{u,e}(t,x,y,\vp):= (J^{u,e}_1(t,x,y,\vp), \cdots, J^{u,e}_I(t,x,y,\vp))^{\top}, \nonumber \\
 \mathscr{L}^{u}\vp(t,x):=\vp_t(t,x)+ \mu_{X}^{\top}(t,x,u)D\vp(t,x)+\frac{1}{2}\text{Tr}[\sigma_X\sigma_X^{\top}(t,x,u)D^2\vp(t,x)]. 
 \end{array}
  \end{equation*}
\begin{remark}
For simplicity, we denote $H^*(t,x,\vp(t,x),D\vp(t,x), D^2\vp(t,x),\vp)$ by $H^*\vp(t,x)$ for $\vp\in C^{1,2}(\D)$. For $\vp\in C^2(\R^d)$, we denote $H^*(T,x,\vp(x),D\vp(x), D^2\vp(x),\vp)$ by $H^*\vp(x)$. We will use similar notation for $H_*, F^{*}, F_{*}$ and other operators in later sections.
\end{remark}

\begin{definition}[Concatenation]
Let $\nu_{1}, \nu_{2} \in \Uc^t_{\unco}$ (resp. $\Uc^t_{\co}$ ), $\tau \in \mathcal{T}_t$. The concatenation of $\nu_1$ and $ \nu_2 $ at  $\tau$ is defined as $
\nu_1\otimes_{\tau} \nu_2 := \nu_1 \mathbbm{1}_{[0,\tau[}+ \nu_2 \mathbbm{1}_{[\tau,T]} \in\Uc_{\unco}^{t}$ (resp. $\Uc^t_{\co}$).\footnote{This can be easily checked.}
\end{definition}  
We will carry out Perron's method to study the stochastic target problems with a non-cooperative stopper and a cooperative stopper, respectively, in Section \ref{sec:superhedging} and \ref{sec:subhedging}.

\section{Analysis of $u_{\unco}$ defined in \eqref{eq: value_function_super}}\label{sec:superhedging}
In this section, we use stochastic Perron's method to prove that 
an appropriate upper envelope of a class of carefully defined functions is a viscosity sub-solution of
\begin{equation}\label{eq: sub_HJB equation_interior-superhedging}
\min\{\vp(t,x)-g(x), -\partial_t\vp(t,x)+ H_*\vp(t,x)\}\leq 0 \;\;\text{in}\;\;\Di.
 \end{equation} 
and an appropriate lower envelope of $u_{\unco}$ is a viscosity super-solution of
 \begin{equation}\label{eq: super_HJB equation_interior-superhedging}
\min\{\vp(t,x)-g(x), -\partial_t\vp(t,x)+ H^*\vp(t,x)\}\geq 0 \;\;\text{in}\;\;\Di
\end{equation}
The boundary conditions will be discussed in Theorem \ref{thm: bd_viscosity_property-superhedging}. These envelopes will be defined in terms of the collections of stochastic super- and sub-solutions that we will define next.
\begin{definition} [Stochastic super-solutions] \label{def: Stochasticsuper-solution-super}
A continuous function $w: \D \rightarrow \mathbb{R}$ is called a stochastic super-solution of \eqref{eq: super_HJB equation_interior-superhedging} if
\begin{enumerate}
\item $w(t, x)\geq g(x)$ and for some $C>0$ and $n\in\N$,\footnote{$C$ and $n$ may depend on $w$ and $T$. This also applies to Definition \ref{def: Stochasticsub-solution-super}, \ref{def: Stochasticsuper-solution_subhedging} and \ref{def: Stochasticsub-solution_subhedging}.} $|w(t,x)|\leq C(1+|x|^{n})$ for all $(t,x)\in \D$.
\item Given $(t,x,y)\in \D\times\mathbb{R}$, for any $\tau\in\mathcal{T}_t$ and $\nu\in \Uc^t_{\unco}$, there exists $\tilde {\nu}\in \Uc^t_{\unco}$ such that  
$$Y(\rho )\geq w(\rho, X(\rho )) \quad \mathbb{P}-\text{a.s.} \text{ on } \{ Y(\tau)\geq w(\tau, X(\tau)) \}$$ for all $\rho \in \mathcal{T}_{\tau}$,
where $X:= X_{t,x}^{\nu\otimes_{\tau}\tilde{\nu}}$ and $Y:=Y_{t,x,y}^{\nu\otimes_{\tau}\tilde{\nu}}$.
\end{enumerate}
\end{definition}
\begin{definition} [Stochastic sub-solutions] \label{def: Stochasticsub-solution-super}
A continuous function $w: \D \rightarrow \mathbb{R}$ is called a stochastic sub-solution of \eqref{eq: sub_HJB equation_interior-superhedging} if
\begin{enumerate}
\item $w(T, x)\leq g(x)$ and for some $C>0$ and $n\in\N$, $|w(t,x)|\leq C(1+|x|^{n})$ for all $(t,x)\in \D$.
\item Given $(t,x,y)\in \D\times\mathbb{R}$, for any $\tau\in\mathcal{T}_t$, $\nu\in \Uc_{\unco}^t$ and $B\subset \{Y(\tau)<w(\tau,X(\tau))\}$ satisfying $B\in\mathcal{F}_\tau^t$ and $\mathbb{P}(B)>0$, there exists $\rho\in\mathcal{T}_{\tau}$ such that
$$\mathbb{P}(Y(\rho )< g(X(\rho ))|B)>0.$$
Here, we use the notation $X:= X_{t,x}^{\nu}$ and $Y:=Y_{t,x,y}^{\nu}$.
\end{enumerate}
\end{definition}
\noindent Denote the sets of stochastic super-solutions and sub-solutions by $\bU^+_{\unco}$ and $\bU^-_{\unco}$, respectively.
\begin{assumption}\label{assump:semisolution_not_empty_super}
$\bU^+_{\unco}$ and $\bU^-_{\unco}$ are not empty. 
\end{assumption}
We will provide sufficient conditions which guarantee Assumption \ref{assump:semisolution_not_empty_super} in the Appendix A. These conditions will be useful once we analyze the zero-sum controller-stopper game. 

We are now ready to define the envelopes we mentioned above.
\begin{definition}
Let $u^+_{\unco}:=\inf_{w \in \mathbb{U}^{+}_{\unco}} w$ and $u^-_{\unco}:=\sup_{w \in \mathbb{U}^{-}_{\unco}} w$.
\end{definition}

\begin{remark}
\begin{itemize}
\item For any stochastic super-solution $w$, choose $\tau=t$. Then there exists $\tilde{\nu}\in \Uc_{\unco}^t$ such that
$Y_{t,x,y}^{\tilde{\nu}}(\rho)\geq w\left(\rho, X_{t,x}^{\tilde{\nu}}(\rho)\right)\geq g\left(X_{t,x}^{\tilde{\nu}}(\rho)\right) \mathbb{P}-\text{a.s. for all } \rho\in\mathcal{T}_{t}
 \text{ if }  y\geq w(t,x).$
Hence, $y\geq w(t,x)$ implies that $y\geq u_{\unco}(t,x)$ from \eqref{eq: value_function_super}. This means that $w\geq u_{\unco}$ and $u^+_{\unco}\geq u_{\unco}$. By the definition of $\bU^{+}_{\unco}$, we know that $u^+_{\unco}(t,x)\geq g(x)$ for all $(t,x)\in\D$.

\item  For any stochastic sub-solution $w$, if $y<w(t,x)$,  by choosing $\tau=t$, we get from the second property of Definition \ref{def: Stochasticsub-solution-super} that for any $\nu\in \Uc^t$,
$\mathbb{P}\left(Y_{t,x,y}^{\nu}(\rho)<  g(X_{t,x}^{\nu}(\rho))\right)> 0
$ for some $\rho\in\mathcal{T}_{t}$.
Therefore, from \eqref{eq: value_function_super}, $y<w(t,x)$ implies that $y\leq u_{\unco}(t,x)$. This means that $w\leq u_{\unco}$ and $u^-_{\unco}\leq u_{\unco}$.  By the definition of $\bU^{-}_{\unco}$, it holds that $u^{-}_{\unco}(T,x)\leq g(x)$ for all $x\in\R^d$.
\end{itemize}
\end{remark}
In short,
\begin{equation}\label{eq:intfvmavp-superhedging}
u^-_{\unco} = \sup _{w\in \mathbb{U}^-_{\unco}} w\leq  u_{\unco} \leq \inf _{w\in \mathbb{U}^+_{\unco}}w= u_{\unco}^+.
\end{equation}

\subsection{Viscosity Property in $\Di$}\label{subsec:interior}
As in \cite{Moreau}, the proof of the sub-solution property requires a regularity assumption on the set-valued map $\mathcal{N}_{0,\eta}(\cdot, \psi)$.
\begin{assumption}\label{assump: regularity-super}
For $\psi\in C(\D)$, $\eta>0$, let $B$ be a subset of $\D\times\R\times\R^d$ such that $\mathcal{N}_{0,\eta}(\cdot, \psi)\neq\emptyset$ on $B$. Then for every $\eps>0$, $(t_0,x_0,y_0,p_0)\in Int(B)$ and $u_0\in\mathcal{N}_{0, \eta}(t_0,x_0,y_0,p_0,\psi) $, there exists an open neighborhood $B'$ of $(t_0,x_0,y_0,p_0)$ and a locally Lipschitz continuous map $\hat{\nu}$ defined on $B'$ such that $\|\hat{\nu}(t_0,x_0,y_0,p_0)-u_0\|_{U}\leq \eps$ and $\hat{\nu}(t,x,y,p)\in \mathcal{N}_{0, \eta}(t,x,y,p, \psi)$.
\end{assumption}  
The following two lemmas can be easily checked. Hence, we omit the proofs. 
\begin{lemma}
$\bU^{+}_{\unco}$ and $\bU^{-}_{\unco}$ are closed under pairwise minimization and maximization, respectively. That is, 
\begin{enumerate}
\item if $w_1, w_2\in \mathbb{U}^+_{\unco}$, then $w_1\wedge w_2\in \mathbb{U}^+_{\unco}$;\; \item if $w_1,w_2\in \mathbb{U}^-_{\unco}$, then $w_1\vee w_2\in \mathbb{U}^-_{\unco}$.
\end{enumerate}
\end{lemma}
\begin{lemma}\label{lem: monotone seq approaches v+ or v_--superhedging}
There exists a non-increasing sequence $\{w_{n}\}_{n=1}^{\infty}\subset\bU^{+}_{\unco}$ such that $w_n\searrow u^+_{\unco}$ and a non-decreasing sequence $\{v_{n}\}_{n=1}^{\infty}\subset\bU^{-}_{\unco}$ such that $v_n\nearrow u^-_{\unco}$.
\end{lemma}
\begin{theorem} \label{thm: main theorem_interior-superhedging}
Under Assumptions \ref{assump: lambda intensity kernel}, \ref{assump: regu_on_coeff}, \ref{assump:semisolution_not_empty_super} and \ref{assump: regularity-super}, 
$u^+_{\unco}$ is an upper semi-continuous (USC) viscosity sub-solution of \eqref{eq: sub_HJB equation_interior-superhedging}. On the other hand, under Assumptions \ref{assump: lambda intensity kernel}, \ref{assump: regu_on_coeff} and \ref{assump:semisolution_not_empty_super}, $u^-_{\unco}$ is a lower semi-continuous (LSC) viscosity super-solution of \eqref{eq: super_HJB equation_interior-superhedging}.
\end{theorem}
\begin{proof} 
Since the proof of the viscosity sub-solution of $u_{\unco}^+$ is as  similar to Step 1 in the proof of Theorem 3.1 in \cite{BayraktarLi-Jump}, we will only show below that $u^-_{\unco}$ is a viscosity super-solution.

\textbf{Step A:} We show in this step that $u^{-}_{\unco}(t,x)\geq g(x)$ for all $(t,x)\in\D$.  Assume, on the contrary, that for some $(t_{0},x_{0})\in\D$, there exists $\eta>0$ such that
\begin{equation}\label{eq:superhedging-interior-u-minus-step2A-contra}
2\eta=g(x_{0})-u^{-}_{\unco}(t_{0},x_{0})>0.
\end{equation}
Choose an arbitrary $w\in\bU_{\unco}^{-}$. By the definition of $\bU_{\unco}^{-}$ and lower semi-continuity of $g$, there exists $\eps>0$ such that 
\begin{equation*}
g(x)-w(t,x)>\eta, \; g(x)-g(x_{0})>-\frac{\eta}{2}, \; |w(t,x)-w(t_{0},x_{0})|\leq \frac{\eta}{2}\text{ for all } (t,x)\in\text{cl}(B_{\eps}(t_{0},x_{0})).
\end{equation*}
Define 
\begin{equation}
w'(t,x):=
\left \{ 
\begin{split}
&w(t,x)+(g(x_{0})-\eta-w(t_{0},x_{0}))\left(1-\text{dist}((t,x), (t_{0},x_{0}))/\eps\right) &\text{ for } (t,x)\in\text{cl}(B_{\eps}(t_0, x_0)),\\
&w(t,x)  & \text{ for } (t,x)\notin\text{cl}(B_{\eps}(t_0, x_0)).
\end{split}
\right.
\end{equation}
Obviously, $w'\geq w$ and $w'$ is continuous with polynomial growth. In addition,  
\begin{equation}\label{eq:superhedging-interior-u-minus-step2A-1}
\{(t,x):w(t,x)<w'(t,x)\}=B_{\eps}(t_{0},x_{0}) \quad \text{and}
\end{equation}
\begin{equation}\label{eq:superhedging-interior-u-minus-step2A-2}
w'(t,x)\leq w(t,x)+ (g(x_{0})-\eta-w(t_{0},x_{0}))< g(x_{0})-\frac{\eta}{2}<g(x) \text{ for }(t,x)\in\text{cl}(B_{\eps}(t_{0},x_{0})).
\end{equation}
The equation above, along with the fact that $w\in\bU^{-}_{\unco}$, implies that $w'(T,x)\leq g(x)$ for all $x\in\R^{d}$. Noting that $w'(t_{0},x_{0})=g(x_{0})-\eta>u_{\unco}^{-}(t_{0},x_{0})$ due to \eqref{eq:superhedging-interior-u-minus-step2A-contra}, we would obtain a contradiction if we could show $w'\in\bU^{-}_{\unco}$.  

To prove that $w'\in\bU^{-}_{\unco}$, fix $(t,x,y)\in \D_i\times\mathbb{R}$, $\tau \in \mathcal{T}_t$ and $\nu\in\mathcal{U}_{\unco}^{t}$. For $w\in\bU_{\unco}^{-}$, let $\rho^{w,\tau,\nu} \in \mathcal{T}_{\tau}$ be the ``optimal'' stopping time satisfying the second item in  Definition \ref{def: Stochasticsub-solution-super}. In order to show that $w'\in\bU^-_{\unco}$, we want to construct an ``optimal'' stopping time $\rho$ which works in the sense of Definition \ref{def: Stochasticsub-solution-super}. Let $A=\{w(\tau,X(\tau))=w'(\tau,X(\tau))\}\in\mathcal{F}^t_{\tau}$ and
\begin{equation*}
\rho=\mathbbm{1}_{A}\rho^{w,\tau,\nu}+\mathbbm{1}_{A^c}\tau.
\end{equation*}
Obviously, $\rho\in\mathcal{T}_{\tau}$. It suffices to show 
$
\mathbb{P}(Y(\rho)<g(X(\rho))|B)>0
$
for any $B\subset \{Y(\tau)<w'(\tau,X(\tau))\}$ satisfying $\mathbb{P}(B)>0$ and $B\in\mathcal{F}^t_{\tau}$. The following two scenarios together will yield the desired result. 

(i) If $\mathbb{P}(B\cap A)>0:$ We know that $B\cap A \subset \{Y(\tau)<w(\tau,X(\tau))\}$ and $B\cap A \in \mathcal{F}^t_{\tau}$. From the fact $w\in\bU^{-}_{\unco}$ and the definition of $\rho$ on $A$, it holds that
\begin{equation*}\label{eq:case1_super-solution-property-superhedging}
\mathbb{P}(Y(\rho)<g(X(\rho))|B\cap A)=\mathbb{P}(Y(\rho^{w,\tau,\nu})<g(X(\rho^{w,\tau,\nu}))|B\cap A)>0.
\end{equation*}

(ii) If $\mathbb{P}(B\cap A^c)>0$: $(\tau,X(\tau))\in B_{\eps}(t_0,x_0)$ on $A^{c}$ from \eqref{eq:superhedging-interior-u-minus-step2A-1}, which implies $w'(\tau,X(\tau))<g(X(\tau))$ from \eqref{eq:superhedging-interior-u-minus-step2A-2}. Since $\rho=\tau$ on $A^{c}$,
\begin{equation*}\label{eq:case2_super-solution-property-superhedging}
\mathbb{P}(Y(\rho)<g(X(\rho))|B\cap A^c)\geq \mathbb{P}(Y(\tau)<w'(\tau,X(\tau))|B\cap A^c) = \mathbb{P}(B\cap A^c)>0.
\end{equation*}
\textbf{Step B:} We claim that $u^{-}_{\unco}$ is a viscosity super-solution of 
\begin{equation*}
-\partial_t\vp(t,x)+ H^*\vp(t,x)\geq 0. 
\end{equation*}
The proof is similar to the proof in Step 2 of Theorem 3.1 in \cite{BayraktarLi-Jump}, but it is worth pointing out the following difference: after $w^{\kappa}$ is defined, we need to construct an optimal stopping time $\rho$ for $w^{\kappa}$ given $\tau$ and $\nu$ (as we did in Step A). In fact, it is easy to see that $\rho$ can be defined as follows:
\begin{equation*}
\rho=\mathbbm{1}_{A} \rho^{w,\tau,\nu} +\mathbbm{1}_{A^{c}} \rho^{w, \theta, \nu}, 
\end{equation*}
where $\rho^{w,\tau,\nu}$ (resp. $\rho^{w,\theta,\nu}$) is the ``optimal'' stopping time in Definition \ref{def: Stochasticsub-solution-super} for $w$ given $\tau$(resp. $\theta$) and $\nu$. Here $\theta$ is the same as that in Step 2 of Theorem 3.1. 
\end{proof}
\subsection{Boundary Conditions}\label{subsec: bdd cond}
By the definition of $u_{\unco }$, it holds that $u_{\unco}(T,x)=g(x)$ for all $x\in\R^d$. However, $u^+_{\unco}$ and $u^-_{\unco}$ may not satisfy this boundary condition. Define
\begin{equation*}
\mathbf{N}(t,x,y,p, \psi):=\{(r,s)\in\R^d\times\R: \exists u \in U, \;\text{s.t.}\; r=N^{u}(t,x,y,p)\;\text{and}\;s\leq \Delta^{u,e}(t,x,y,\psi)\; \hat{m}-\text{a.s.}\}
\end{equation*}
and 
\begin{equation}\label{eq:delta defi}
\delta:=\text{dist}(0, \mathbf{N}^c)-\text{dist}(0, \mathbf{N}),
\end{equation}
where dist denotes the Euclidean distance. It holds that
\begin{equation}\label{eq: delta>0 equi int}
0\in\text{int}(N(t,x,y,p,\psi))\;\;\text{iff}\;\;\delta(t,x,y,p,\psi)>0. 
\end{equation}
We refer the readers to \cite{Moreau} for the discussion of the boundary conditions. 

The upper (resp. lower) semi-continuous envelope of $\delta$ is denoted by $\delta^* \;(\text{resp}.\; \;\delta_*)$. Let
\begin{equation*}\label{eq: boundary_u-_def}
u^+_{\unco}(T-,x)=\limsup _{(t<T, x')\rightarrow (T,x)} u^-_{\unco}(t,x'), \;\;u^-_{\unco}(T-,x)=\liminf _{(t<T, x')\rightarrow (T,x)} u^-_{\unco}(t,x').
\end{equation*}
The following theorem is an adaptation of Theorem 4.1 in \cite{BayraktarLi-Jump}. 
\begin{theorem}\label{thm: bd_viscosity_property-superhedging}
Under Assumptions \ref{assump: lambda intensity kernel}, \ref{assump: regu_on_coeff}, \ref{assump:semisolution_not_empty_super} and \ref{assump: regularity-super},  $u^{+}_{\unco}(T-,\cdot)$ is a USC viscosity sub-solution of 
$\min\{\vp(x)-g(x), \delta_{*}\vp(x)\}\leq 0\text{ on } \R^d.$
On the other hand, under Assumptions \ref{assump: lambda intensity kernel}, \ref{assump: regu_on_coeff} and \ref{assump:semisolution_not_empty_super}, $u^{-}_{\unco}(T-,\cdot)$ is an LSC viscosity super-solution of 
$
\min\{\vp(x)-g(x), \delta^{*}\vp(x) \}\geq 0\text{ on }\R^d.
$
\end{theorem}

\section{Zero-sum Controller-Stopper Game}\label{subsec:equivalence}
\noindent In this section we show that the HJB equation associated to a stochastic controller-stopper game can be deduced from a stochastic target problem with a non-cooperative stopper. Given a bounded continuous function $g:\R^d\rightarrow\R$, we define  a stochastic controller-stopper game by
$$
 \mathbf{u}_{\unco}(t,x):=\inf_{\nu\in\mathcal{U}^t}\sup_{\rho\in\mathcal{T}_{t}}\mathbb{E}[g(X_{t,x}^{\nu}(\rho))]. 
$$
We follow the setup of Section \ref{sec:prob} with one exception:
$\mathcal{U}^{t}$ is the collection of all the $\mathbb{F}^{t}$-predictable processes in $\mathbb{L}^2(\Omega\times[0,T], \mathcal{F}\otimes\mathcal{B}[0,T], \mathbb{P}\otimes \lambda_{L}; U)$, where $U\subset \R^{d}$ and $X$ follows the SDE
\begin{equation*}
dX(s)=\mu_{X}(s,X(s),\nu(s))ds+\sigma_{X}(s,X(s),\nu(s))dW_s+\int_{E} \beta(s,X(s-),\nu(s), e)\lambda(ds,de). 
\end{equation*}
The following embedding lemma is an adaptation of a result in \cite{Bouchard_Equivalence}. 

\begin{lemma}\label{eq:equivalence_application-superhedging}
Suppose Assumptions \ref{assump: lambda intensity kernel} holds. Define
\begin{equation*}
\begin{gathered}
u_{\unco}(t,x):=\inf\{y\in\R: \exists (\nu, \alpha, \gamma)\in \mathcal{U}^t\times\mathcal{A}^t\times\Gamma^t_{\unco}\;\text{s.t.}\;Y_{t,y}^{\alpha,\gamma}(\rho)\geq g(X_{t,x}^{\nu}(\rho))\;\text{for all } \rho\in\mathcal{T}_{t}\}, \text{where} \\
Y_{t,y}^{\alpha,\gamma}(\cdot):=y+\int_t^{\cdot}\alpha^{\top}(s)dW_s+\int_t^{\cdot}\int_E\gamma^{\top}(s,e)\tilde{\lambda}(ds,de), 
\end{gathered}
\end{equation*}
and $\mathcal{A}^t$ and $\Gamma^{t}_{\unco}$ are the collections of $\R^{d}$-valued and $\mathbb{L}^{2}(E, \mathcal{E}, \hat{m}; \R^{I})$-valued processes, respectively,  satisfying the the measurability and the integrablity condition \eqref{eq:integrablility} in Section \ref{sec:prob}. 
Then $u_{\unco}=\mathbf{u}_{\unco}$ on $\D$.
\end{lemma}
\begin{proof}
For fixed $\nu\in\mathcal{U}^{t}$, let 
\begin{equation*}
A^{\nu}(s):=\esssup_{\tau\in\mathcal{T}_s}\mathbb{E}[g(X^{\nu}_{t,x}(\tau))|\mathcal{F}_s], s \geq t.
\end{equation*}
Then $A^{\nu}$ is the Snell envelope (starting at $t$) of $g(X_{t,x}^{\nu})$ and thus a super-martingale. Moreover, 
\begin{equation*}
\esssup_{\tau\in\mathcal{T}_t}\mathbb{E}[g(X^{\nu}_{t,x}(\tau))|\mathcal{F}_t] + A^{\nu}(\rho)-A^{\nu}(t)\geq g(X^{\nu}_{t,x}(\rho))\text{ for all }\rho\in\mathcal{T}_{t}.
\end{equation*}
By Doob-Meyer Decomposition Theorem, $A^{\nu}_{s}=M^{\nu}_{s}-C^{\nu}_{s}$ for $s\in[t,T]$, where $M^{\nu}$ is a martingale on $[t,T]$ and $C^{\nu}$ is an increasing adapted process with $C^{\nu}_{t}=0$. Therefore, 
\begin{equation*}
\esssup_{\tau\in\mathcal{T}_t}\mathbb{E}[g(X^{\nu}_{t,x}(\tau))|\mathcal{F}_t] + M^{\nu}(\rho)-M^{\nu}(t)\geq g(X^{\nu}_{t,x}(\rho))\text{ for all }\rho\in\mathcal{T}_{t}.
\end{equation*}
Denote $\mathcal{M}_{\unco}=\{M^{\nu}: \nu\in\Uc^{t}\}$. In view of Lemma \ref{lem:stochastic_target_representation_superhedging},  it suffices to check that
\begin{equation}\label{eq:inclusion_equi}
\mathcal{M}_{\unco}\subset
\mathcal{M}:=\left\{Y_{t,y}^{\alpha,\gamma}(\cdot): y\in\R, \alpha\in\mathcal{A}^{t}, \gamma\in\Gamma^{t}_{\unco} \right\}.
\end{equation}
In fact, by the martingale representation theorem (see e.g. Theorem 14.5.7 in \cite{MR3443368}), for  any $\nu\in\mathcal{U}^{t}$, $M^{\nu}$ can be represented in the form of $Y_{t,y}^{\alpha, \gamma}$ for some $\alpha \in\mathcal{A}^{t}$ and $\gamma\in\Gamma^{t}_{0}$, where  $\Gamma_{0}^{t}$ is the collection of $\mathbb{L}^{2}(E, \mathcal{E}, \hat{m}; \R^{I})$-valued processes satisfying all of the admissibility conditions except for $\eqref{eq:admissibility_super}$. We now prove that $\Gamma_{0}^{t}$ in the claim above can be actually replaced by $\Gamma^{t}_{\unco}$. Assume, contrary to \eqref{eq:inclusion_equi}, that there exists $\nu_{0}\in\mathcal{U}^{t}$ such that
\begin{equation*}
\mathcal{M}^{\nu_{0}}(\cdot)=y+\int_t^{\cdot}\alpha_{0}^{\top}(s)dW_s+\int_t^{\cdot}\int_E\gamma_{0}^{\top}(s,e)\tilde{\lambda}(ds,de)
\end{equation*}
for some $y\in\mathbb{R}$,  $\alpha_{0}\in\mathcal{A}^{t}$ and $\gamma_{0}\in\Gamma^{t}_{0}$, but  \eqref{eq:admissibility_super} does not hold.  This means that for $K>2\|g\|_{\infty}$, there exists $\tau_{0}\in\mathcal{T}_{t}$ such that 
$\mathbb{P}\left(\int_{E}\gamma_{0}^{\top}(\tau_{0},e) \lambda(\{\tau_{0}\},de)\leq -K\right)>0. $
Therefore, $$M^{\nu_{0}}(\tau_{0})-M^{\nu_{0}}(\tau_{0}-) = \int_{E}\gamma_{0}^{\top}(\tau_{0},e) \lambda(\{\tau_{0}\},de)\leq -K\;\;\text{with positive probability},$$
which further implies that
$$
A^{\nu_{0}}(\tau_{0})-A^{\nu_{0}}(\tau_{0}-)\leq -K \;\;\text{with positive probability}.
$$
This contradicts the fact that $A^{\nu_{0}}$ is (strictly) bounded by $\frac{K}{2}$. 
\end{proof}
Let $\mathbf{H}^*$ be the USC envelope of the LSC map $\mathbf{H}:\D\times\R^d\times\mathbb{M}^d\times C(\D) \rightarrow \R$ defined by
\begin{equation*}
\begin{array}{c}
\mathbf{H}: (t,x,p,A, \vp) \rightarrow \sup_{u\in U}\{-I[\vp](t,x,u)-\mu_X^{\top}(t,x,u)p-\frac{1}{2}\text{Tr}[\sigma_X\sigma_X^{\top}(t,x,u)A]\},\;\text{where} \\
I[\vp](t,x,u)=\sum_{1\leq i\leq I}\int_E \left( \vp(t,x+\beta_i(t,x,u,e))-\vp(t,x)\right)m_i(de).
\end{array}
\end{equation*}
\begin{theorem}\label{thm: optimal control-superhedging}
Under Assumptions \ref{assump: lambda intensity kernel} and \ref{assump: regu_on_coeff}, $u^+_{\unco}$  is a USC viscosity sub-solution of 
$$
\min\{\vp(t,x)-g(x), -\partial_t\vp(t,x)+\mathbf{H}\vp(t,x)\}\leq 0\text{ on } \Di
$$
and $u^+_{\unco}(T-,x)\leq g(x)$ for all $x\in \R^d$. On the other hand, $u^-_{\unco}$  is an LSC viscosity super-solution of 
$$
\min\{\vp(t,x)-g(x), -\partial_t\vp(t,x)+\mathbf{H}^{*}\vp(t,x)\}\geq 0\text{ on } \Di
$$
and $u^-_{\unco}(T-,\cdot)\geq g(x)$ for all $x\in\R^{d}$.
\end{theorem}
\begin{proof}
It is easy to check Assumption \ref{assump: regularity-super} for the stochastic target problem. Since $g$ is bounded, we can check that all of the assumptions in the Appendix A are satisfied, which implies that Assumption \ref{assump:semisolution_not_empty_super} holds. From Theorem \ref{thm: main theorem_interior-superhedging}, 
$u^+_{\unco}$ is a USC viscosity sub-solution of 
$$
\min\{\vp(t,x)-g(x), -\partial_t\vp(t,x)+H_{*}\vp(t,x)\}\leq 0\text{ on } \Di
$$
and $u^-_{\unco}$ is an LSC viscosity super-solution of 
$$
\min\{\vp(t,x)-g(x), -\partial_t\vp(t,x)+H^{*}\vp(t,x)\}\geq 0\text{ on } \Di
$$
From Proposition 3.3 in \cite{Bouchard_Equivalence}, 
$H^*\leq \mathbf{H}^*$ and $H_*\geq \mathbf{H}$. This implies that the viscosity properties in the parabolic interior hold. From the definition of $\delta$ in \eqref{eq:delta defi}, we know that
\begin{equation*}
\begin{array}{ll}
\mathbf{N}(t,x,y,p,\vp)=&\{(q,s)\in\R^d\times\R: \exists (u,a,r)\in U\times\R^d\times\L^2(E,\mathcal{E}, \hat{m};\R^I)\;\text{s.t.}
\;q=a-\sigma_X^{\top}(t,x,u)p \\ &\text{and }s\leq \min_{1\leq i\leq I}\{r_i(e)-\vp(t,x+\beta_i(t,x,u,e))+\vp(t,x)\} \; \hat{m}-\text{a.s.}\; e\in E\; \}.
\end{array}
\end{equation*}
Obviously, $\mathbf{N}=\R^d\times\R$. Therefore, $\delta=\infty$ and the boundary conditions hold. 
\end{proof}
\noindent The following two corollaries show that $\mathbf{u}_{\unco}$ is the unique viscosity solution to its associated HJB equation. We omit the proofs as the proofs are relatively simple given the above result.
\begin{corollary}\label{coro2-superhedging}
Suppose that Assumptions \ref{assump: lambda intensity kernel} and \ref{assump: regu_on_coeff} hold,  $\mathbf{H}=\mathbf{H}^*$ on $\{\mathbf{H}<\infty\}$ and there exists an LSC function $\mathbf{G}:\D\times\R\times\R^d\times\mathbb{M}^d\times C(\D)\cap \{\mathbf{H}<\infty\}\rightarrow\R $  such that 
\begin{equation*}
\begin{array}{c}
(a)\; \mathbf{H}(t,x,y,p,M,\vp)<\infty \implies \mathbf{G}(t,x,y,p,M,\vp)\leq 0, \\
(b)\;\mathbf{G}(t,x,y,p,M,\vp)<0\implies \mathbf{H}(t,x,y,p,M,\vp)<\infty.
\end{array}
\end{equation*} 
Then $u^+_{\unco}$ $(\text{resp. } u^-_{\unco})$ is a USC $($resp. an LSC$)$ viscosity sub-solution $($resp. super-solution$)$ of
\begin{equation*}
\min\{\vp(t,x)-g(x), \max\{-\partial_t\vp(t,x)+\mathbf{H}\vp(t,x), \mathbf{G}\vp(t,x)\}\}=0\;\;\text{on}\;\;\D_{i}.
\end{equation*}  
\end{corollary}
\begin{remark}\label{Remark:existence of G}
Consider that $b = \beta = 0$ and the state process $X$ follows the set-up in \cite{Bayraktar_and_Sirbu_SP_HJBEqn}. 
In the case of one-dimensional utility maximization, where $H(t,x,p,M) = p/2M^{2}$, one can see that $H(t,x,\cdot)$ is continuous and finite in $\R\times \R$ but not at $(0,0)$. Then we can easily check that $G= - e^{-H}$ satisfies all the properties in Corollary 4.1. 
\end{remark}

\begin{corollary}
Suppose that all the assumptions in Corollary~\ref{coro2-superhedging} hold. Then 
$u^{+}_{\unco}(T-,x)=u^{-}_{\unco}(T-,x)=g(x)$. 
Moreover, if the comparison principle holds for 
\begin{equation*}
\min\{\vp(t,x)-g(x), \max\{-\partial_t\vp(t,x)+\mathbf{H}\vp(t,x), \mathbf{G}\vp(t,x)\}\}=0\;\;\text{on}\;\;\D_{i},
\end{equation*}
then $u_{\unco}$$(=\mathbf{u}_{\unco})$ is the unique continuous viscosity solution with $u_{\unco}(T,x)=g(x)$.
\end{corollary}
\begin{remark}
As for the assumptions needed for the comparison principle to hold, we refer the readers to Theorem 4.1 in \cite{Pham1998} for a similar comparison principle result. Thus, we get an example in which the comparison principle holds (up to slight modification). A more general result for controlled jumps is provided in \cite{Barles2008}.
\end{remark}

\section{Analysis of $u_{\co}$ defined in \eqref{eq: value_function_sub}}\label{sec:subhedging}
In the section, using stochastic Perron's method we prove that an appropriate upper bound of $u_{\co}$ 
 is a viscosity sub-solution of
\begin{equation}\label{eq: sub_HJB equation_interior-subhedging}
\min\{\vp(t,x)-g(x), -\partial_t\vp(t,x)+ F_*\vp(t,x)\}\leq 0 \;\;\text{in}\;\;\Di.
 \end{equation} 
and an appropriate lower bound is a viscosity super-solution of 
 \begin{equation}\label{eq: super_HJB equation_interior-subhedging}
\min\{\vp(t,x)-g(x), -\partial_t\vp(t,x)+ F^*\vp(t,x)\}\geq 0 \;\;\text{in}\;\;\Di
\end{equation}
 The boundary conditions will be deferred to Theorem \ref{thm: bd_viscosity_property_subhedging}. In order to construct the aforementioned upper and lower envelopes we will introduce two classes of functions next.
\begin{definition} [Stochastic super-solutions] \label{def: Stochasticsuper-solution_subhedging}
A continuous function $w: \D \rightarrow \mathbb{R}$ is called a stochastic super-solution of \eqref{eq: super_HJB equation_interior-subhedging} if
\begin{enumerate}
\item $w(t, x)\geq g(x)$ and for some $C>0$ and $n\in\N$, $|w(t,x)|\leq C(1+|x|^{n})$ for all $(t,x)\in \D$.
\item Given $(t,x,y)\in \D\times\mathbb{R}$, for any $\tau\in\mathcal{T}_t$, $\rho \in \mathcal{T}_{\tau}$ and $\nu\in \Uc_{\co}^{t}$,   we have 
$$\mathbb{P}(Y(\rho )>w(\rho, X(\rho))|B)>0$$
for any $B\subset \{Y(\tau)>w(\tau,X(\tau))\}$ satisfying $B\in\mathcal{F}_\tau^t$ and $\mathbb{P}(B)>0$. Here, $X:= X_{t,x}^{\nu}$ and $ Y:=Y_{t,x,y}^{\nu}$.
\end{enumerate}
\end{definition}
\begin{definition} [Stochastic sub-solutions] \label{def: Stochasticsub-solution_subhedging}
A continuous function $w: \D \rightarrow \mathbb{R}$ is called a stochastic sub-solution of \eqref{eq: sub_HJB equation_interior-subhedging} if
\begin{enumerate}
\item  $w(T, x)\leq g(x)$ for all $x\in\mathbb{R}^d$ and for some $C>0$ and $n\in\N$, $|w(t,x)|\leq C(1+|x|^{n})$ for all $(t,x)\in \D$.
\item Given$(t,x,y)\in \D\times\mathbb{R}$, for any $\tau\in\mathcal{T}_t$ and $\nu\in \Uc_{\co}^t$, there exist $\rho \in \mathcal{T}_{\tau}$ and $\tilde{\nu}\in\mathcal{U}_{\co}^t$ such that
$$Y(\rho )\leq g(X(\rho)) \text{ on } \{Y(\tau)\leq w(\tau,X(\tau))\},$$
where $X:= X_{t,x}^{\nu\otimes_{\tau}\tilde{\nu}}$ and $ Y:=Y_{t,x,y}^{\nu\otimes_{\tau}\tilde{\nu}}$.
\end{enumerate}
\end{definition}
Denote the sets of stochastic super-solutions and sub-solutions by $\bU^+_{\co}$ and $\bU^-_{\co}$, respectively.
\begin{assumption}\label{assump:semisolution_not_empty_sub}
$\bU^+_{\co}$ and $\bU^-_{\co}$ are not empty. 
\end{assumption}
Sufficient conditions for the above assumption are given in Appendix A. These conditions will be useful once we analyze the cooperative controller-stopper problem. 
We are ready to define the aforementioned envelopes.
\begin{definition}
Let $u^+_{\co}:=\inf_{w \in \mathbb{U}^{+}_{\co}} w$ and $u^-_{\co}:=\sup_{w \in \mathbb{U}^{-}_{\co}} w$.
\end{definition}

\begin{remark}
\begin{itemize}
\item  For $w\in\bU^{+}_{\co}$, choose $\tau=t$. Then for any $\nu\in\Uc^{t}_{\co}$ and $\rho\in\mathcal{T}_{t}$, it holds that
$\P(Y(\rho) >g\left(X(\rho)\right))>\P(Y(\rho)>w(\rho, X(\rho)))>0 
 \text{ if }  y>w(t,x).$
Hence, $y\geq w(t,x)$ implies that $y\geq u_{\co}(t,x)$ from \eqref{eq: value_function_sub}. This means that $w\geq u_{\co}$ and $u^+_{\co}\geq u_{\co}$. By the definition of $\bU^{+}_{\co}$, we know that $u^+_{\co}(t,x)\geq g(x)$ for all $(t,x)\in\D$.
\item For $w\in\bU^{-}_{\co}$, if $y\leq w(t,x)$,  by choosing $\tau=t$, we get that there exist $\tilde{\nu}\in\bU^{t}_{\co}$ and $\rho\in\mathcal{T}_{t}$ such that
$Y_{t,x,y}^{\tilde{\nu}}(\rho)\leq g(X_{t,x}^{\tilde\nu}(\rho))\; \P\as.$
Therefore, from \eqref{eq: value_function_sub}, $y<w(t,x)$ implies that $y\leq u(t,x)$. This means that $w\leq u_{\co}$ and $u^-_{\co}\leq u_{\co}$.  By the definition of $\bU^{-}_{\co}$, it holds that $u^{-}_{\co}(T,x)\leq g(x)$ for all $x\in\R^d$.
\end{itemize}
\end{remark}
In short,
\begin{equation}\label{eq:intfvmavp-subhedging}
u^-_{\co} = \sup _{w\in \mathbb{U}^-_{\co}} w\leq  u_{\co} \leq \inf _{w\in \mathbb{U}^+_{\co}}w= u_{\co}^+.
\end{equation}

\subsection{Viscosity Property in $\Di$}
Before we state the main results, we need the following assumption which is crucial to the super-solution property of $u^{+}_{\co}$.
\begin{assumption}\label{assump: regularity_sub}
For $\psi\in C(\D)$, $\eta>0$, let $B$ be a subset of $\D\times\R\times\R^d$ such that $\mathcal{M}_{0,-\eta}(\cdot, \psi)\neq\emptyset$ on $B$. Then for every $\eps>0$, $(t_0,x_0,y_0,p_0)\in Int(B)$ and $u_0\in\mathcal{M}_{0, -\eta}(t_0,x_0,y_0,p_0,\psi) $, there exists an open neighborhood $B'$ of $(t_0,x_0,y_0,p_0)$ and a locally Lipschitz continuous map $\hat{\nu}$ defined on $B'$ such that $\|\hat{\nu}(t_0,x_0,y_0,p_0)-u_0\|_{U}\leq \eps$ and $\hat{\nu}(t,x,y,p)\in \mathcal{M}_{0, -\eta}(t,x,y,p, \psi)$.
\end{assumption}  
As before we have the following two results whose proofs will be omitted.
\begin{lemma}
$\bU^{+}_{\co}$ and $\bU^{-}_{\co}$ are closed under pairwise minimization and maximization, respectively. 
\end{lemma}
\begin{lemma}\label{lem: monotone seq approaches v+ or v_--subhedging}
There exists a non-increasing sequence $\{w_{n}\}_{n=1}^{\infty}\subset\bU^{+}_{\co}$ such that $w_n\searrow u^+_{\co}$ and a non-decreasing sequence $\{v_{n}\}_{n=1}^{\infty}\subset\bU^{-}_{\co}$ such that $v_n\nearrow u^-_{\co}$.
\end{lemma}
\begin{theorem} \label{thm: main theorem_interior-subhedging}
Under Assumptions \ref{assump: lambda intensity kernel}, \ref{assump: regu_on_coeff}, \ref{assump:semisolution_not_empty_sub} and \ref{assump: regularity_sub},
$u^+_{\co}$ is an upper semi-continuous (USC) viscosity sub-solution of \eqref{eq: sub_HJB equation_interior-subhedging}. On the other hand, under Assumptions \ref{assump: lambda intensity kernel}, \ref{assump: regu_on_coeff} and \ref{assump:semisolution_not_empty_sub}, $u^-_{\co}$ is a lower semi-continuous (LSC) viscosity super-solution of \eqref{eq: super_HJB equation_interior-subhedging}.
\end{theorem}
\begin{proof} 
{\bf Step 1 ($u_{\co}^+$ is a viscosity sub-solution).} The proof of this claim is similar to Step 2 of the proof of Theorem 3.1 in \cite{BayraktarLi-Jump}. The difference is that the proof uses sub-martingale property since the target is $Y \leq g(X)$ instead of $Y \ge g(X)$.

{\bf \noindent Step 2 ($u^-_{\co}$ is a viscosity super-solution).} \\
\textbf{Step A:} We show in this step that $u^{-}_{\co}(t,x)\geq g(x)$ for all $(t,x)\in\D$.  Assume, on the contrary, that for some $(t_{0},x_{0})\in\D$, there exists $\eta>0$ such that
\begin{equation}\label{eq:subhedging-interior-u-minus-step2A-contra}
2\eta=g(x_{0})-u^{-}_{\co}(t_{0},x_{0})>0.
\end{equation}
Choose an arbitrary $w\in\bU_{\co}^{-}$. By the definition of $\bU_{\co}^{-}$ and lower semi-continuity of $g$, there exists $\eps>0$ such that 
\begin{equation*}
g(x)-w(t,x)>\eta, \; g(x)-g(x_{0})>-\frac{\eta}{2}, \; |w(t,x)-w(t_{0},x_{0})|\leq \frac{\eta}{2}\text{ for all } (t,x)\in\text{cl}(B_{\eps}(t_0,x_0)).
\end{equation*}
Define 
\begin{equation*}
w'(t,x):=
\left \{ 
\begin{split}
&w(t,x)+(g(x_{0})-\eta-w(t_{0},x_{0}))\left(1-\text{dist}((t,x), (t_{0},x_{0}))/\eps\right) &\text{ for } (t,x)\in\text{cl}(B_{\eps}(t_0, x_0)),\\
&w(t,x)  & \text{ for } (t,x)\notin\text{cl}(B_{\eps}(t_0, x_0)).
\end{split}
\right.
\end{equation*}
Obviously, $w'\geq w$ and $w'$ is continuous with polynomial growth. In addition,  
\begin{equation}\label{eq:subhedging-interior-u-minus-step2A-1}
\{(t,x):w(t,x)<w'(t,x)\}=B_{\eps}(t_{0},x_{0}) \quad \text{and}
\end{equation}
\begin{equation}\label{eq:subhedging-interior-u-minus-step2A-2}
w'(t,x)\leq w(t,x)+ (g(x_{0})-\eta-w(t_{0},x_{0}))< g(x_{0})-\frac{\eta}{2}<g(x) \text{ for }(t,x)\in\text{cl}(B_{\eps}(t_{0},x_{0})).
\end{equation}
The equation above, along with the fact that $w\in\bU^{-}_{\co}$, implies that $w'(T,x)\leq g(x)$ for all $x\in\R^{d}$. Noting that $w'(t_{0},x_{0})=g(x_{0})-\eta>u_{\co}^{-}(t_{0},x_{0})$ due to \eqref{eq:subhedging-interior-u-minus-step2A-contra}, we would obtain a contradiction if we could show $w'\in\bU^{-}_{\co}$. We now prove that $w'\in\bU^{-}_{\co}$. 

Fix $(t,x,y)\in \D_i\times\mathbb{R}$, $\tau \in \mathcal{T}_t$ and $\nu\in\mathcal{U}_{\co}^{t}$. For $w\in\bU_{\co}^{-}$, let $\rho^{w,\tau,\nu} \in \mathcal{T}_{\tau}$ and $\tilde{\nu}^{w,\tau,\nu}$ be the ``optimal'' stopping time and control satisfying the second item in Definition \ref{def: Stochasticsub-solution_subhedging}. In order to show that $w'\in\bU^-_{\co}$, we want to construct an ``optimal'' stopping time $\rho$ and ``optimal'' control $\tilde{\nu}$ which work for $w'$ in the sense of Definition \ref{def: Stochasticsub-solution_subhedging}. Let $A=\{w(\tau,X(\tau))=w'(\tau,X(\tau))\}\in\mathcal{F}^t_{\tau}$, 
\begin{equation*}
\rho=\mathbbm{1}_{A}\rho^{w,\tau,\nu}+\mathbbm{1}_{A^c}\tau \;\text{and}\; \tilde{\nu}=(\mathbbm{1}_{A}\tilde{\nu}^{w,\tau,\nu}+\mathbbm{1}_{A^c} u_{0})\mathbbm{1}_{[\tau, T]}, 
\end{equation*}
where $u_{0}$ is an arbitrary element in $U$. Obviously, $\rho\in\mathcal{T}_{\tau}$ and $\tilde{\nu}\in\mathcal{T}_{t}$. It suffices to show 
$$
Y(\rho)\leq g(X(\rho))\;\; \P\as\;\;\text{on}\;\;\{Y\leq w'(\tau, X(\tau))\}.
$$
(i) On $A\cap\{Y\leq w'(\tau, X(\tau))\}$: Note that $A\cap\{Y\leq w'(\tau, X(\tau))\} \subset \{Y(\tau)\leq w(\tau,X(\tau))\}$.  From the fact $w\in\bU^{-}_{\co}$ and the definition of $\rho$ and $\tilde{\nu}$ on $A$, it holds that
\begin{equation}\label{eq:case1_super-solution-property-subhedging}
Y(\rho)=Y(\rho^{w,\tilde{\nu},\tau})\leq g(X(\rho^{w,\tilde{\nu},\tau}))= g(X(\rho)) \text{ on } A\cap\{Y\leq w'(\tau, X(\tau))\}.
\end{equation}
(ii) On $A^{c}\cap\{Y\leq w'(\tau, X(\tau))\}$: $(\tau,X(\tau))\in B_{\eps}(t_0,x_0)$ on $A^{c}$ from \eqref{eq:subhedging-interior-u-minus-step2A-1}, which implies $w'(\tau,X(\tau))<g(X(\tau))$ from \eqref{eq:subhedging-interior-u-minus-step2A-2}. This, together with the fact that $\rho=\tau$ on $A^{c}$, implies that 
\begin{equation}\label{eq:case2_super-solution-property-subhedging}
Y(\rho)\leq w'(\rho, X(\rho))\leq g(X({\rho}))\text{  on  }A^{c}\cap\{Y\leq w'(\tau, X(\tau))\}. 
\end{equation}
\textbf{Step B:} 
We claim that $u^{-}_{\co}$ is a viscosity super-solution to 
$$
-\partial_t\vp(t,x)+ F^*\vp(t,x) \geq 0. 
$$
We omit this proof, which is rather long, in the interest of space. This follows the outline of Step 1 in the proof ofTheorem 3.1 of \cite{BayraktarLi-Jump}. It is worth noting that
after the construction of $w^{\kappa}$ in that proof, given $(t,x,y)\in \D_i\times\mathbb{R}$, $\tau \in \mathcal{T}_t$ and $\nu\in\mathcal{U}_{\co}^{t}$, we need to construct an ``optimal'' stopping time $\rho$ and ``optimal'' control $\tilde{\nu}$ which work for $w'$ in the sense of Definition \ref{def: Stochasticsub-solution_subhedging} (as we did above in Step A).
\end{proof}

\subsection{Boundary condition for $u_{\co}$}
As for the boundary conditions, instead of studying $u_{\co}^+(T,x)$ and $u_{\co}^-(T,x)$, we still consider
\begin{equation*}\label{eq: boundary_u-_def_subhedging}
u_{\co}^+(T-,x)=\limsup _{(t<T, x')\rightarrow (T,x)} u_{\co}^+(t,x'), \;\;u_{\co}^-(T-,x)=\liminf _{(t<T, x')\rightarrow (T,x)} u_{\co}^-(t,x').
\end{equation*}
\begin{theorem}\label{thm: bd_viscosity_property_subhedging}
Under Assumptions \ref{assump: lambda intensity kernel}, \ref{assump: regu_on_coeff}, \ref{assump:semisolution_not_empty_sub} and \ref{assump: regularity_sub}, $u_{\co}^{+}(T-,\cdot)$ is an USC viscosity sub-solution of 
\begin{equation*}
\left(\vp(x)-g(x)\right)\mathbbm{1}_{\{F_*\vp(x)>-\infty\}}\leq 0\text{ on } \R^d. 
\end{equation*}
Moreover,
$
u_{\co}^{-}(T-,x)\geq g(x) \text{ for all }x\in\R^d.
$
\end{theorem}
\begin{proof}
Since $u_{\co}^-(t,x)\geq g(x)$ for any $(t,x)\in\D$ due to Step 2A of Theorem \ref{thm: main theorem_interior-subhedging}, it directly follows that $u^-_\co(T-,x)\geq g(x)$. The proof of the subsolution property is longer but the proof of Theorem 4.1 in \cite{BayraktarLi-Jump} can be adapted to the present case. 
\end{proof}

\section{Cooperative Controller-Stopper Game}\label{sec:equivalence-subhedging}
\noindent In this section, we prove that a cooperative controller-stopper problem can expressed in terms of a stochastic target problem with a cooperative stopper. Given a bounded continuous function $g:\R^d\rightarrow\R$, we define
$$
 \mathbf{u}_{\co}(t,x):=\sup_{\nu\in\mathcal{U}^t}\sup_{\rho\in\mathcal{T}_{t}}\mathbb{E}[g(X_{t,x}^{\nu}(\rho))]. 
$$
We follow the setup of Section \ref{sec:prob} with one exception:
$\mathcal{U}^{t}$ is the collection of all the $\mathbb{F}^{t}$-predictable processes in $\mathbb{L}^2(\Omega\times[0,T], \mathcal{F}\otimes\mathcal{B}[0,T], \mathbb{P}\otimes \lambda_{L}; U)$, where $U\subset \R^{d}$ and $X$ follows the SDE
\begin{equation*}
dX(s)=\mu_{X}(s,X(s),\nu(s))ds+\sigma_{X}(s,X(s),\nu(s))dW_s+\int_{E} \beta(s,X(s-),\nu(s), e)\lambda(ds,de). 
\end{equation*}

\begin{lemma}\label{eq:equivalence_application}
Suppose Assumptions \ref{assump: lambda intensity kernel} and \ref{assump: regu_on_coeff} hold. Define a stochastic target problem as follows:
\begin{equation*}
\begin{gathered}
u_{\co}(t,x):=\sup\{y\in\R: \exists (\nu, \alpha, \gamma)\in \mathcal{U}^t\times\mathcal{A}^t\times\Gamma^t_{\co}\text{ and } \rho\in\mathcal{T}_{t}\;\text{s.t.}\;Y_{t,y}^{\alpha,\gamma}(\rho)\leq g(X_{t,x}^{\nu}(\rho))\}, \text{where} \\
Y_{t,y}^{\alpha,\gamma}(\cdot):=y+\int_t^{\cdot}\alpha^{\top}(s)dW_s+\int_t^{\cdot}\int_E\gamma^{\top}(s,e)\tilde{\lambda}(ds,de) 
\end{gathered}
\end{equation*}
and $\mathcal{A}^t$ and $\Gamma^{t}_{\co}$ are the sets of $\R^{d}$-valued and $\mathbb{L}^{2}(E, \mathcal{E}, \hat{m}; \R^{I})$-valued processes, respectively,  satisfying the admissibility conditions in Section \ref{sec:prob}. 
Then $u_{\co}=\mathbf{u}_{\co}$ on $\D$.
\end{lemma}
\begin{proof}
In view of Lemma \ref{lem:stochastic_target_representation_subhedging} and Remark \ref{re:subhedging}, it suffices to check that
\begin{equation}\label{eq:inclusion_equi-subhedging}
\mathcal{M}_{\co}\subset
\mathcal{M}:=\left\{Y_{t,y}^{\alpha,\gamma}(\cdot): y\in\R, \alpha\in\mathcal{A}^{t}, \gamma\in\Gamma^{t} \right\},
\end{equation}
where $\mathcal{M}_{\co}$ is defined as in Remark  \ref{re:subhedging}. In fact, by the martingale representation theorem, for  any $\nu\in\mathcal{U}^{t}$ and $\rho\in\mathcal{T}_{t}$, $\mathbb{E}[g(X_{t,x}^{\nu}(\rho))|\mathcal{F}^{t}_{\cdot}]$ can be represented in the form of $Y_{t,y}^{\alpha, \gamma}$ for some $\alpha \in\mathcal{A}^{t}$ and $\gamma\in\Gamma^{t}_{0}$,\footnote{Such $\alpha$ and $\gamma$ are unique} where  $\Gamma_{0}^{t}$ is the set of $\mathbb{L}^{2}(E, \mathcal{E}, \hat{m}; \R^{I})$-valued processes satisfying all of the admissibility conditions except $\eqref{eq:admissibility_sub}$. We now prove that such $\gamma$ satisfies the condition in \eqref{eq:admissibility_sub}, thus finishing the proof.

Assume, contrary to \eqref{eq:inclusion_equi-subhedging}, that there exists $\nu_{0}\in\mathcal{U}^{t}$ and $\rho\in\mathcal{T}_{t}$ such that
\begin{equation*}
\mathbb{E}[g(X_{t,x}^{\nu_{0}}(\rho))|\mathcal{F}^{t}_{\cdot}]=y+\int_t^{\cdot}\alpha_{0}^{\top}(s)dW_s+\int_t^{\cdot}\int_E\gamma_{0}^{\top}(s,e)\tilde{\lambda}(ds,de)
\end{equation*}
for some $y\in\mathbb{R}$,  $\alpha_{0}\in\mathcal{A}^{t}$ and $\gamma_{0}\in\Gamma^{t}_{0}$, but  \eqref{eq:admissibility_sub} does not hold for $\gamma_{0}$. In the equation above, $\mathbb{E}[g(X_{t,x}^{\nu_{0}}(\rho))|\mathcal{F}^{t}_{\cdot}]$ can be chosen to be c\`adl\`ag, thanks to Theorem 1.3.13 in \cite{ShreveKaratzas}. Then for $K>2\|g\|_{\infty}$, there exists $\tau_{0}\in\mathcal{T}_{t}$ such that 
$\mathbb{P}\left(\int_{E}\gamma_{0}^{\top}(\tau_{0},e) \lambda(\{\tau_{0}\},de)>K\right)>0.$ Letting $M_{0}(\cdot)=\mathbb{E}\left[g(X_{t,x}^{\nu_{0}}(\rho))|\mathcal{F}^{t}_{\cdot}\right]$, we get that $$M_{0}(\tau_{0})-M_{0}(\tau_{0}-) = \int_{E}\gamma_{0}^{\top}(\tau_{0},e) \lambda(\{\tau_{0}\},de)>K\;\;\text{with positive probability}.$$
Since $|M_{0}|$ is bounded by $\|g\|_{\infty}<K/2$, we obtain a contradiction. 
\end{proof}
Let $\mathbf{F}_{*}$ be the LSC envelope of the USC map $\mathbf{F}:\D\times\R^d\times\mathbb{M}^d\times C(\D) \rightarrow \R$ defined by
\begin{equation*}
\begin{array}{c}
\mathbf{F}: (t,x,p,A, \vp) \rightarrow \inf_{u\in U}\{-I[\vp](t,x,u)-\mu_X^{\top}(t,x,u)p-\frac{1}{2}\text{Tr}[\sigma_X\sigma_X^{\top}(t,x,u)A]\},\;\text{where} \\
I[\vp](t,x,u)=\sum_{1\leq i\leq I}\int_E \left( \vp(t,x+\beta_i(t,x,u,e))-\vp(t,x)\right)m_i(de).
\end{array}
\end{equation*}
\begin{theorem}\label{thm: optimal control-subhedging}
Under Assumptions \ref{assump: lambda intensity kernel} and \ref{assump: regu_on_coeff}, $u^+_{\co}$  is a USC viscosity sub-solution of 
$$
\min\{\vp(t,x)-g(x), -\partial_t\vp(t,x)+\mathbf{F}_{*}\vp(t,x)\}\leq 0\text{ on } \Di
$$
and $u^+_{\co}(T-,x)$ is an USC viscosity sub-solution of 
\begin{equation*}
\left(\vp(x)-g(x)\right)\mathbbm{1}_{\{\mathbf{F}_{*}\vp(x)>-\infty\}}\leq 0\text{ on } \R^d. 
\end{equation*}
On the other hand, $u^-_{\co}$  is an LSC viscosity super-solution of 
$$
\min\{\vp(t,x)-g(x), -\partial_t\vp(t,x)+\mathbf{F}\vp(t,x)\}\geq 0\text{ on } \Di
$$
and $u^-_{\co}(T-,\cdot)\geq g(x)$ for all $x\in\R^{d}$.
\end{theorem}
\begin{proof}
It is easy to check Assumption \ref{assump: regularity_sub} for the stochastic target problem. Since $g$ is bounded, we can check that all of the assumptions in the Appendix A are satisfied, which implies that Assumption \ref{assump:semisolution_not_empty_sub}
 holds. From Theorem \ref{thm: main theorem_interior-subhedging}, 
$u^+_{\co}$ is a USC viscosity sub-solution of 
$$
\min\{\vp(t,x)-g(x), -\partial_t\vp(t,x)+F_{*}\vp(t,x)\}\leq 0\text{ on } \Di
$$
and $u^-_{\co}$ is an LSC viscosity super-solution of 
$$
\min\{\vp(t,x)-g(x), -\partial_t\vp(t,x)+F^{*}\vp(t,x)\}\geq 0\text{ on } \Di
$$
From Proposition 3.3 in \cite{Bouchard_Equivalence}, 
$F^*\leq \mathbf{F}$ and $F_*\geq \mathbf{F}_{*}$. Thus, we get our desired results. 
\end{proof}
The following two corollaries (whose proofs are omitted) show that $\mathbf{u}_{\co}$ is the unique viscosity solution to its associated HJB equation. 
\begin{corollary}\label{coro2}
Suppose that Assumptions \ref{assump: lambda intensity kernel} and \ref{assump: regu_on_coeff} hold,  $\mathbf{F}=\mathbf{F}_{*}$ on $\{\mathbf{F}>-\infty\}$ and there exists a USC function $\mathbf{G}:\D\times\R\times\R^d\times\mathbb{M}^d\times C(\D)\cap \{\mathbf{F}>-\infty\}\rightarrow\R$ such that 
\begin{equation*}
\begin{array}{c}
(a)\; \mathbf{F}(t,x,y,p,M,\vp)>-\infty \implies \mathbf{G}(t,x,y,p,M,\vp)\geq 0, \\
(b)\;\mathbf{G}(t,x,y,p,M,\vp)>0\implies \mathbf{F}(t,x,y,p,M,\vp)>-\infty.
\end{array}
\end{equation*} 
Then $u^+_{\co}$ $(\text{resp. } u^-_{\co})$ is a USC $($resp. an LSC$)$ viscosity sub-solution $($resp. super-solution$)$ of
\begin{equation*}
\min\{\vp(t,x)-g(x), \max\{-\partial_t\vp(t,x)+\mathbf{F}\vp(t,x), \mathbf{G}\vp(t,x)\}\}=0\;\;\text{on}\;\;\D_{i}.
\end{equation*}  
\end{corollary}
\begin{remark}
A remark similar to Remark \ref{Remark:existence of G} applies regarding the verifiability of the assumption above. 
\end{remark}
\begin{corollary}
Suppose that all the assumptions in Corollary \ref{coro2} hold. Additionally, assume that there is a comparison principle between USC sub-solutions and LSC super-solutions for the PDE
\begin{equation}\label{eq:bd_pde_control}
\min\{\vp(x)-g(x), \mathbf{G}\vp(x)\}=0\;\;\text{on}\;\;\R^{d}.
\end{equation}
Then 
$u^{+}_{\co}(T-,x)=u^{-}_{\co}(T-,x)=\hat{g}(x)$, where $\hat{g}$ is the unique continuous viscosity solution to \eqref{eq:bd_pde_control}. 
Moreover, if the comparison principle holds for 
\begin{equation}
\min\{\vp(t,x)-g(x), \max\{-\partial_t\vp(t,x)+\mathbf{F}\vp(t,x),\; \mathbf{G}\vp(t,x)\}\}=0 \text{ on }\Di,
\end{equation}
then $u_{\co}$$(=\mathbf{u}_{\co})$ is the unique continuous viscosity solution with $u_{\co}(T,x)=\hat{g}(x)$.
\end{corollary}
\begin{remark}
To get a comparison principle, we can adopt the proof in \cite{Pham1998} appropriately like in \cite{Bayraktar_Huang}.
\end{remark}
\appendix
\section*{Appendix A}\label{sec:appendix}
\renewcommand{\thesection}{A}
\renewcommand{\thetheorem}{A.\arabic{theorem}}
\noindent  We provide sufficient conditions for the nonemptiness of $\bU^+_{\unco}$, $\bU^-_{\unco}$, $\bU^{+}_{\co}$ and $\bU^{-}_{\co}$. 
\begin{assumption}\label{assump: g bounded}
$g$ is bounded.
\end{assumption}
 \begin{assumption} \label{assump: existence of no_investing_strategy}
 There exists $u_0 \in U$ such that $\sigma_Y(t,x,y,u_0)=0$ and $b(t,x,y,u_0(e),e)=0$ for all $(t,x,y,e)\in \D\times \mathbb{R}\times E$.
 \end{assumption}
 \begin{proposition}\label{prop: U^+ is not empty}
 Under Assumptions  \ref{assump: lambda intensity kernel}, \ref{assump: regu_on_coeff}, \ref{assump: g bounded} and \ref{assump: existence of no_investing_strategy}, $\bU^{+}_{\unco}$ and $\bU^{-}_{\co}$ are not empty.
 \end{proposition}
\begin{proof} We will only show $\bU^{+}_{\unco}$ is not empty. A very similar proof applies to $\bU^{-}_{\co}$. \\
\noindent\textbf{Step 1.}
In this step we assume that $\mu_{Y}$ is non-decreasing in its $y$-variable. We will show that $w(t,x)=\gamma-e^{kt}$ is a stochastic super-solution for some choice of $k$ and $\gamma$. 

By the linear growth condition on $\mu_Y$ in Assumption \ref{assump: regu_on_coeff}, there exists $L>0$ such that $$|\mu_Y(t,x,y,u_0)|\leq  L(1+|y|),$$ where $u_0$ is the element in $U$ in Assumption \ref{assump: existence of no_investing_strategy}. Choose $k\geq 2L$ and $\gamma$ such that $-e^{k T}+\gamma\geq \|g\|_{\infty}$. 
 Then $w(t,x)\geq w(T,x)\geq g(x)$ for all $(t,x)\in\D$. It suffices to show that for any $(t,x,y)\in \D\times\mathbb{R}$, $\tau\in\mathcal{T}_t$, $\nu\in \Uc^t_{\unco}$ and $\rho\in \mathcal{T}_{\tau}$,
 \begin{equation}\label{eq: property of stochastic super-solution}
  Y(\rho )\geq w(\rho, X(\rho )) \;\; \mathbb{P}\text{-a.s.}\;\; \text{on}\;\;\{Y(\tau )\geq w(\tau, X(\tau))\}, \text{where } X:= X_{t,x}^{\nu\otimes_{\tau}u_0}, Y:=Y_{t,x,y}^{\nu\otimes_{\tau}u_0}.
 \end{equation}
Let $A= \{Y(\tau )> w(\tau, X(\tau))\}$, $V(s)=w(s,X(s))$ and $\Gamma(s)=\left(V(s)-Y(s)\right)\mathbbm{1}_{A}.$ Therefore, for $s\geq \tau$, 
\begin{gather}\label{eq: Gamma_integral}
dY(s)= \mu_{Y}\left(s,X(s),Y(s), u_0\right)ds, \;dV(s)= -ke^{ks}ds, \;
\Gamma(s)=\mathbbm{1}_{A}\int_{\tau}^{s} ( \xi(q)+ \Delta(q) ) dq + \mathbbm{1}_{A}\Gamma(\tau) , \text{where} \\
\Delta(s):=-ke^{ks}-\mu_Y(s,X(s),V(s),u_0)\leq -ke^{ks}-\mu_Y(s,X(s),-e^{ks},u_0)\leq -ke^{ks}+L(1+e^{ks})\leq 0, \nonumber\\
\xi(s):=\mu_{Y}(s,X(s),V(s),u_0) -\mu_{Y}(s,X(s),Y(s),u_0). \nonumber
\end{gather}
Therefore, from \eqref{eq: Gamma_integral} and the definitions of $\Gamma$ and $A$, it holds that
$$
\Gamma(s)\leq \mathbbm{1}_A\int_{\tau}^{s} \xi(q) dq \;\; \text{and}\;\; \Gamma^{+}(s)\leq \mathbbm{1}_A\int_{\tau}^{s} \xi^{+}(q) dq 
 \;\; \text{for} \;\; s\geq \tau.$$ 
From the Lipschitz continuity of $\mu_Y$ in $y$-variable in Assumption \ref{assump: regu_on_coeff}, 
\begin{equation*}\label{eq: fun_gronwall}
\Gamma^{+}(s)\leq \mathbbm{1}_A \int_{\tau}^{s} \xi^{+}(q) dq \leq \int_{\tau}^{s} L_0 \Gamma^{+}(q) dq  \;\; \text{for} \;\; s\geq \tau,
\end{equation*} 
where $L_0$ is the Lipschitz constant of $\mu_Y$ with respect to $y$. Note that we use the assumption that $\mu_Y$ is non-decreasing in its $y$-variable to obtain the second inequality.
Since $\Gamma^+(\tau)=0$, an application of Gr\"{o}nwall's Inequality implies that
 $\Gamma^+(\rho)\leq 0$, which further implies that \eqref{eq: property of stochastic super-solution} holds. \\
\textbf{Step 2.} We get rid of our assumption on $\mu_{Y}$ from Step 1 by following a proof similar to those in \cite{BayraktarLi} and \cite{Bouchard_Nutz_TargetGames}. For $c>0$, define $\widetilde Y_{t,x,y}^{\nu}$ as the strong solution of
\begin{equation*}
\begin{split}
  d\widetilde{Y}(s)&=\tilde \mu_{Y}(s,X_{t,x}^{\nu}(s),\widetilde{Y}(s),\nu(s)) ds +\tilde \sigma_{Y}^{\top}(s,X_{t,x}^{\nu}(s),\widetilde{Y}(s),\nu(s))dW_{s} \\ &+ \int_{E} \widetilde{b}^{\top}(s,X_{t,x}^{\nu}(s-),\widetilde{Y}(s-), \nu_1(s),\nu_2(s,e), e)\lambda(ds,de)
  \end{split}
\end{equation*}
  with initial data $\widetilde{Y}(t)=y$, where
  \begin{equation*}
  \begin{array}{c}
   \widetilde{\mu}_{Y}(t,x,y,u):= c y+e^{ct} \mu_{Y}(t,x,e^{-c t}y,u), \;
  \widetilde{\sigma}_{Y}(t,x,y,u):= e^{c t} \sigma_{Y}(t,x,e^{-c t} y,u), \\
  \widetilde{b}(t,x,y,u(e),e):= e^{c t} b(t,x,e^{-c t} y,u(e),e). 
  \end{array}
  \end{equation*}
 Therefore,  
 \begin{equation}\label{eq:trans to increasing}
\widetilde{Y}_{t,x,y}^{\nu}(s)e^{-cs}=Y_{t,x,ye^{-ct}}^{\nu}(s), \;t\leq s\leq T. 
 \end{equation}
 Let 
  \begin{equation}\label{eq:new super stg}
  \tilde u_{\unco}(t,x)= \inf\{y\in \R: \exists \; \nu\in \mathcal{U}^t_{\unco}, \mbox{  s.t.}\; \widetilde{Y}^{\nu}_{t,x,y}(\rho)\ge \tilde g(\rho, X^{\nu}_{t,x}(\rho))\;\as\},
  \end{equation}
 where $\tilde{g}(t, x)=e^{ct} g(x)$. Therefore, from \eqref{eq:trans to increasing}, $\tilde{u}_{\unco}(t,x)=e^{ct}u_{\unco}(t,x).$ Since $\mu_{Y}$ is Lipschitz in $y$, we can choose $c>0$ so that
$
 \widetilde {\mu}_{Y}: (t,x,y,u) \mapsto  cy +  e^{c t}\mu_{Y}(t,x,e^{-c t}y,u)
$
is non-decreasing in $y$. Moreover, all the properties of $\widetilde{\mu}_{Y}, \widetilde{\sigma}_{Y}$ and $\widetilde{b}$ in Assumption \ref{assump: regu_on_coeff} still hold. We replace $\mu_Y$, $\sigma_Y$ and $b$ in all of the equations and definitions in Section \ref{sec:prob} with  $\widetilde{\mu}_{Y}, \widetilde{\sigma}_{Y}$ and $\widetilde{b}$, we get $\widetilde{H}^*$ and $\widetilde{H}_*$. Let $\widetilde{\bU}^+_{\unco}$ be the set of stochastic super-solutions of the new target problem \eqref{eq:new super stg}. It is easy to see that $w\in\bU^+_{\unco}$ if and only if $\widetilde{w}(t,x):=e^{ct}w(t,x)\in \widetilde{\bU}^+_{\unco}$. From Step 1, $\widetilde{\bU}^+_{\unco}$ is not empty. Thus, $\bU^+_{\unco}$ is not empty. 
\end{proof}
 \begin{assumption}\label{assump: linear growth in y}
 There is $C\in \R$ such that for all $(t,x,y,u,e)\in\D\times\R\times U\times E$,
 $$\left|\mu_Y(t,x,y,u)+\int_E b^{\top}(t,x,y,u(e),e) m(de)\right|\leq C(1+|y|).$$
 \end{assumption}
\begin{proposition}\label{Prop: U^- is not empty}
Under Assumptions \ref{assump: lambda intensity kernel}, \ref{assump: regu_on_coeff}, \ref{assump: g bounded} and \ref{assump: linear growth in y}, $\mathbb{U}_{\unco}^{-}$ and $\bU^{+}_{\co}$ are not empty.
\end{proposition}
 \begin{proof} We will only show that $\mathbb{U}_{\unco}^{-}$ is not empty. Assume that $$\mu_{Y}(t,x,y,u)+\int_E b^{\top}(t,x,y,u(e),e)m(de)$$ is non-decreasing in its $y$-variable. We could remove this assumption by using the argument from previous proposition. 

Choose $k\geq 2C$ ($C$ is the constant in Assumption \ref{assump: linear growth in y}) and $\gamma>0$ such that $e^{k T}-\gamma<- \|g\|_{\infty}$. Let $w(t,x)=e^{kt}-\gamma$. Notice that $w$ is continuous, has polynomial growth in $x$ and $w(T,x)\leq g(x)$ for all $x\in\R^{d}$. It suffices to show that for any $(t,x,y)\in \D\times\mathbb{R}$, $\tau\in\mathcal{T}_{t}$ and $\nu\in \Uc^t_{\unco}$,  there exists $\rho\in\mathcal{T}_{t}$ such that
$\mathbb{P}(Y(\rho )< g( X(\rho ))|B)>0$
for $B\subset \{Y(\tau)<w(\tau,X(\tau))\}$ satisfying $B\in\mathcal{F}_\tau^t$ and $\P(B)>0$, where $X:= X_{t,x}^{\nu}$ and $Y:=Y_{t,x,y}^{\nu}$. Define
 \begin{equation*}
 \begin{gathered}
M(\cdot)=Y(\cdot)-\int_{\tau}^{\cdot}K(s)ds,\; V(s)=w(s,X(s)),\;
 A= \{Y(\tau )<w(\tau, X(\tau))\},\; \Gamma(s)=\left(Y(s)-V(s)\right)\mathbbm{1}_{A}, \\
\text{where } K(s):=\mu_{Y}(s,X(s),Y(s),\nu(s))+\int_{E}b^{\top}(s,X(s-),Y(s-),\nu_1(s),\nu_2(s,e),e)m(de),\\
 \widetilde{K}(s):=\mu_{Y}(s,X(s),V(s),\nu(s))+\int_{E}b^{\top}(s,X(s-),V(s-),\nu_1(s),\nu_2(s,e),e)m(de).
 \end{gathered}
 \end{equation*}
 It is easy to see that $M$ is a martingale after $\tau.$  Due to the facts that $A\in\mathcal{F}_\tau^t$ and $dV(s)= ke^{ks}ds$, we further know
 \begin{equation}\label{eq: supermar1_nonemptyness of U+}
  \mathbbm{1}_{A}\left(Y(\cdot)-V(\cdot)+\int_{\tau}^{\cdot} ke^{ks}-K(s) ds \right)\;\; \text{is a martingale after}\;\;\tau .
 \end{equation} 
Since Assumption \ref{assump: linear growth in y} holds and $\mu_{Y}(t,x,y,u)+\int_E b^{\top}(t,x,y,u(e),e)m(de)$ is non-decreasing in $y$,
$$
\widetilde{K}(s)\leq \mu_Y(s,X(s),e^{ks}, \nu(s))+\int_{E}b^{\top}(s,X(s-),e^{ks},\nu_1(s),\nu_2(s,e),e)m(de)\leq 2C e^{ks}. 
$$
Therefore, it follows from \eqref{eq: supermar1_nonemptyness of U+}, the inequality above and the fact $k\geq 2C$ that 
\begin{equation}\label{eq: supermar2_nonemptyness of U+}
\widetilde{M}(\cdot):=\mathbbm{1}_{A}\left(Y(\cdot)-V(\cdot)-\int_{\tau}^{\cdot}\xi(s)ds)\right) \;\;\text{is a super-martingale after }\tau,
\end{equation}
where $\xi(s):=K(s)-\widetilde{K}(s)$. Since $\widetilde{M}(\tau)<0$ on $B$, there exists a non-null set $F\subset B $ such that $\widetilde{M}(\rho)<0$ on $F$ for any $\rho\in\mathcal{T}_{\tau}$. By the definition of $\widetilde{M}$ in \eqref{eq: supermar2_nonemptyness of U+}, we get
\begin{equation}\label{eq: Gamma_rho_strict_ineq on F}
\Gamma(\rho)< \mathbbm{1}_{A}\int_{\tau}^{\rho}\xi(s)ds \;\;\text{on}\;\;F.
\end{equation}
 Therefore, 
 \begin{equation}\label{eq: ineq for gronwall ineq}
 \Gamma^{+}(\rho)\leq \mathbbm{1}_A\int_{\tau}^{\rho} \xi^{+}(s) ds 
\leq \int_{\tau}^{\rho} L_0 \Gamma^{+}(s) ds\;\;\text{on}\;\;F. 
\end{equation}
  By Gr\"{o}nwall's Inequality,  $\Gamma^+(\tau)=0$ implies that
  $\Gamma^+(\rho)=0$ on $F$. More precisely, for $\omega\in F$ ($\P-\text{a.s.}$), $\Gamma^{+}(s)(\omega)=0$ for $s\in [\tau(\omega),\rho(\omega)]$. This implies that we can replace the inequalities with equalities in \eqref{eq: ineq for gronwall ineq}. Therefore, by \eqref{eq: Gamma_rho_strict_ineq on F}, $\Gamma(\rho)<0$ on $F$, which yields  
 $\mathbb{P}(Y(\rho )< g( X(\rho ))|B)>0.$ 
\end{proof}

\section*{Appendix B}\label{sec:appendixB}
\renewcommand{\thesection}{B}
\renewcommand{\thetheorem}{B.\arabic{theorem}}
\noindent 
Let $T$ be a finite time horizon, given a general probability space $(\Omega, \mathcal{F},\mathbb{P})$ endowed with a
filtration $\mathbb{F} = \{\mathcal{F}_t\}_{0\leq t \leq T}$ satisfying the usual conditions. Let $\mathcal{T}_t$ be the set of $\mathbb{F}$-stopping times valued in $[t,T]$. In particular, let $\mathcal{T}:=\mathcal{T}_0$. We assume that $\mathcal{F}_0$ is trivial. Let us consider an optimal control problem defined as follows. Let $\mathcal{U}$ be the collection of all $\mathbb{F}$-predictable processes valued in $U\subset\R^k$ and $\{G^{\nu},\nu\in\mathcal{U}\}$ be a collection of bounded, right-continuous processes valued in $\R$. 
Given $(t,\nu)\in[0,T]\times\mathcal{U}$, we consider two optimal stopping control problems:
\begin{equation}
V_{\unco}^{\nu}(t)=\essinf_{\mu\in\mathcal{U}(t,\nu)}\esssup_{\tau\in\mathcal{T}_t}\mathbb{E}[G^{\mu}(\tau)|\mathcal{F}_t], 
\end{equation}
and
\begin{equation}
V_{\co}^{\nu}(t)=\esssup_{\mu\in\mathcal{U}(t,\nu)}\esssup_{\tau\in\mathcal{T}_t}\mathbb{E}[G^{\mu}(\tau)|\mathcal{F}_t], 
\end{equation}
where $\mathcal{U}(t,\nu)=\{\mu\in\mathcal{U}, \mu= \nu\;\text{on}\;[0,t]\;\;\mathbb{P}-\text{a.s.}\}$.
\begin{lemma}\label{lem:stochastic_target_representation_superhedging}
Given $t\in[0,T]$ and $\nu\in\Uc_{t}$, let $\mathcal{M}$ be any family of martingales which satisfies the following:
\begin{equation}\label{eq: property of martinagle set_appendix_superhedging}
\begin{array}{c}
 \text{For any}\;\; \mu\in\mathcal{U}(t,\nu),\;\text{there exists an}\;\; M\in\mathcal{M} \;\text{such that}\;\; \\ \esssup_{\tau\in\mathcal{T}_t}\mathbb{E}[G^{\mu}(\tau)|\mathcal{F}_t] + M(\rho)-M(t)\geq G^{\mu}(\rho)\;\;\text{for all }\rho\in\mathcal{T}_t.
\end{array}
\end{equation}
 Then $V_{\unco}^{\nu}(t)=Y_{\unco}^{\nu}(t),$ where
\begin{equation*}
Y_{\unco}^{\nu}(t)=\essinf\left\{ Y\in L^1(\Omega, \mathcal{F}_t, \mathbb{P})\;|\; \exists (M,\mu)\in \mathcal{M}\times\mathcal{U}(t,\nu), \text{s.t.}\;  Y+M(\rho)-M(t)\geq G^{\mu}(\rho)\;\;\text{for all }\rho\in \mathcal{T}_t \right\}.
\end{equation*}
\end{lemma}
\begin{proof}
(1) $Y_{\unco}^{\nu}(t)\geq V_{\unco}^{\nu}(t)$: Fix $Y\in L^1(\Omega, \mathcal{F}_t, \mathbb{P})$ and $(M,\mu)\in\mathcal{M}\times\mathcal{U}(t,\nu)$ such that $$Y+M(\rho)-M(t)\geq G^{\mu}(\rho) \;\;\text{for all }\rho\in\mathcal{T}_t.$$ By taking the conditional
 expectation, we get that
 \begin{equation*}
Y \geq \mathbb{E}[G^{\mu}(\rho)|\mathcal{F}_t]\;\;\text{for all }\rho\in\mathcal{T}_t. 
 \end{equation*}
 which implies that $Y\geq V_{\unco}^{\nu}(t)$. Therefore, $Y_{\unco}^{\nu}(t)\geq V_{\unco}^{\nu}(t)$. \\
(2)  $V_{\unco}^{\nu}(t)\geq Y_{\unco}^{\nu}(t)$: we get from \eqref{eq: property of martinagle set_appendix_superhedging}, for each $\mu\in\mathcal{U}(t,\nu)$, there exists an $M\in\mathcal{M}$ such that
\begin{equation*}
\esssup_{\tau\in\mathcal{T}_t}\mathbb{E}[G^{\mu}(\tau)|\mathcal{F}_t]+M(\rho)-M(t)\geq G^{\mu}(\rho)\text{ for all }\rho\in\mathcal{T}.
\end{equation*}
This implies that $$\esssup_{\tau\in\mathcal{T}_t}\mathbb{E}[G^{\mu}(\tau)|\mathcal{F}_t]\geq Y_{\unco}^{\nu}(t), $$
 which further implies $V_{\unco}^{\nu}(t)\geq Y_{\unco}^{\nu}(t)$. 
\end{proof}
\begin{lemma}\label{lem:stochastic_target_representation_subhedging}
Let $\mathcal{M}$ be any family of martingales which satisfies the following:
\begin{equation}\label{eq: property of martinagle set_appendix_subhedging}
 \text{For any}\;\; \nu\in\mathcal{U} \;\text{and}\;  \rho\in\mathcal{T},\;\text{there exists an}\;\; M\in\mathcal{M} \;\text{such that}\;\; G^{\nu}(\rho) = M(\rho).
\end{equation}
 Then for each $(t,\nu)\in[0,T]\times\mathcal{U}$, $V_{\co}^{\nu}(t)=Y_{\co}^{\nu}(t),$ where
\begin{equation*}
Y_{\co}^{\nu}(t)=\esssup\left\{ Y\in L^1(\Omega, \mathcal{F}_t, \mathbb{P})\;| \exists (M,\mu,\rho)\in \mathcal{M}\times\mathcal{U}(t,\nu)\times\mathcal{T}_t,\text{s.t.}\;\;Y+M(\rho)-M(t)\leq G^{\mu}(\rho) \right\}.
\end{equation*}
\end{lemma}
\begin{proof}
(1) $Y_{\co}^{\nu}(t)\leq V_{\co}^{\nu}(t)$: Fix $Y\in L^1(\Omega, \mathcal{F}_t, \mathbb{P})$ and $(M,\mu,\rho)\in\mathcal{M}\times\mathcal{U}(t,\nu)\times\mathcal{T}_t$ such that $$Y+M(\rho)-M(t)\leq G^{\mu}(\rho).$$ Then by taking the conditional
 expectation, we get that
 \begin{equation*}
Y \leq \mathbb{E}[G^{\mu}(\rho)|\mathcal{F}_t]\leq V_{\co}^{\nu}(t),
 \end{equation*}
 which implies that $Y_{\co}^{\nu}(t)\leq V_{\co}^{\nu}(t)$. \\
(2)  $Y_{\co}^{\nu}(t)\geq V_{\co}^{\nu}(t)$: we get from \eqref{eq: property of martinagle set_appendix_subhedging}, for each $\mu\in\mathcal{U}(t,\nu)$ and $\rho\in\mathcal{T}_t$, there exists an $M\in\mathcal{M}$ such that
\begin{equation*}
\mathbb{E}[G^{\mu}(\rho)|\mathcal{F}_t]+M(\rho)-M(t)=G^{\mu}(\rho). 
\end{equation*}
In particular,
\begin{equation*}
\mathbb{E}[G^{\mu}(\rho)|\mathcal{F}_t]+M(\rho)-M(t)\leq G^{\mu}(\rho). 
\end{equation*}
Therefore, $\mathbb{E}[G^{\mu}(\rho)|\mathcal{F}_t]\leq Y_{\co}^{\nu}(t)$, which implies $V_{\co}^{\nu}(t)\leq Y_{\co}^{\nu}(t)$.
\end{proof}
\begin{remark}\label{re:subhedging}
It is clear that a collection of martingales which satisfies \eqref{eq: property of martinagle set_appendix_subhedging} always exists. In particular, one can take 
\begin{equation*}
\mathcal{M}_{\co}=\{\{\mathbb{E}[G^{\nu}(\rho)|\mathcal{F}_t]\}_{0\leq t\leq T}, \nu\in\mathcal{U}, \rho\in\mathcal{T}\}. 
\end{equation*}
\end{remark}
{\small
\bibliographystyle{amsplain}
\bibliography{mybib_T}}
\end{document}